\documentclass[11pt]{article}
\usepackage[left=1.25in,right=1.25in, top=1.5in, bottom=1in, includefoot,footskip=30pt,]{geometry}
\usepackage{multirow}
\usepackage{amsmath,amssymb,amsthm,amsfonts}
\usepackage{mathrsfs}
\usepackage{xspace}
\usepackage{algorithm}
\usepackage{algorithmic}
\usepackage{bm}
\usepackage{color}
\usepackage{graphicx}
\usepackage{caption}
\usepackage{subcaption}
\usepackage{multirow}
\usepackage{hhline}
\usepackage[colorlinks,citecolor={blue},urlcolor={black},linkcolor={blue}]{hyperref}
\usepackage{url}

\newcommand{\goto}{\rightarrow}

\graphicspath{{figs/}}

\newcommand\peel{\texttt{peeling}\xspace}

\newcommand\iid{i.\,i.\,d.\xspace}

\newcommand{\R}{\mathbb{R}}
\newcommand{\ep}[1]{\ifthenelse{\equal{#1}{1}}{\mathrm{e}}{\mathrm{e}^{#1}}}

\newcommand{\PBH}{PrivateBHq}
\newcommand{\RNM}{Report Noisy Min}

\let\P\prob
\def\lap{\operatorname{Lap}}
\def\var{\operatorname{Var}}

\def\argmin{\operatorname{argmin}}
\def\range{\operatorname{range}}
\def\fdr{\textnormal{FDR}}
\def\fdp{\textnormal{FDP}}

\let\e\varepsilon
\renewcommand\epsilon{\varepsilon}

\def\d{\mathrm{d}}

\newcommand{\remove}[1]{}

\def\RNMa{Report Noisy Max\xspace}
\def\RNMi{Report Noisy Min\xspace}
\def\BH{BHq\xspace}
\def\BHq{BHq\xspace}
\def\dba{D}
\def\dbb{D'}

\def\eps{\varepsilon}

\newcommand{\E}{\operatorname{\mathbb{E}}}

\newcommand\dbax{D}
\newcommand\dbby{D'}

\newtheorem{theorem}{Theorem}
\newtheorem{othertheorem}{othertheorem}[section]

\newtheorem{lemma}[othertheorem]{Lemma}
\newtheorem{corollary}[othertheorem]{Corollary}
\newtheorem{proposition}[othertheorem]{Proposition}

\theoremstyle{definition}
\newtheorem{definition}[othertheorem]{Definition}
\theoremstyle{remark}

\theoremstyle{assumption}

\theoremstyle{definition}
\newtheorem{example}[othertheorem]{Example}

\let\truenull\pi
\def\pi{\theta}

\allowdisplaybreaks
\numberwithin{equation}{section}

\title{Differentially Private False Discovery Rate Control}
\author{Cynthia Dwork \and Weijie J.~Su \and  Li Zhang}
\date{}

\begin{document}
\maketitle

{\centering
\vspace*{-0.5cm}
Harvard University, University of Pennsylvania, and Google Research
\par\bigskip
\date{}\par
}

\begin{abstract}
Differential privacy provides a rigorous framework for privacy-preserving data analysis. This paper proposes the first differentially private procedure for controlling the false discovery rate (FDR) in multiple hypothesis testing. Inspired by the Benjamini-Hochberg procedure (BHq), our approach is to first repeatedly add noise to the logarithms of the $p$-values to ensure differential privacy and to select an approximately smallest $p$-value serving as a promising candidate at each iteration; the selected $p$-values are further supplied to the BHq and our private procedure releases only the rejected ones. Moreover, we develop a new technique that is based on a backward submartingale for proving FDR control of a broad class of multiple testing procedures, including our private procedure, and both the BHq step-up and step-down procedures. As a novel aspect, the proof works for arbitrary dependence between the true null and false null test statistics, while FDR control is maintained up to
a small multiplicative factor. 

\end{abstract}



\section{Introduction}
\label{sec:intro}

With the growing availability of large-scale datasets, decision-making in healthcare, information technology, and government agencies is increasingly driven by data analyses. This data-driven paradigm, however, comes with great risk if the databases contain sensitive information of individuals such as health records or financial data. Without appropriate adjustments, statistical analysis applied to these databases can lead to privacy violation. For example, Homer et al.~demonstrate that, under certain conditions, it is possible to determine whether an individual with a known genotype is in a genome-wide association study (GWAS) even when only minor allele frequencies are revealed \cite{homer2008resolving}. Such privacy issues have serious implications: at best, individuals and agencies are discouraged from sharing their data for research purposes due to the concern of privacy leakage, impeding scientific progress \cite{kaye2012tension}; at worst, potential adversaries could make use of sensitive information to jeopardize the social foundations of liberal democracy \cite{facebook}.

Being able to conduct data analysis in a way that preserves privacy, therefore, is key to removing barriers to scientific research while preventing breaches of personal data. First introduced by Dwork et al.~\cite{DworkMNS06}, \textit{differential privacy} (Definition \ref{def:dp}) has put private data analysis on a rigorous foundation. A differentially private algorithm is required to hide the presence or absence of any individual or small group of individuals, the intuition being that an adversary unable to tell whether or not a given individual is even a member of the database surely cannot glean information specific to this individual. In computer science, considerable efforts have been made to develop private data release mechanisms \cite{DworkMNS06,mcsherry2007mechanism,barak2007privacy} and private machine learning algorithms under
differential privacy constraints, for example, boosting \cite{hardt2012simple}, empirical risk minimization \cite{chaudhuri2011differentially}, private PAC learning \cite{beimel2010bounds}, and deep learning~\cite{abadi2016deep,bu2020deep}. On the statistical front, differential privacy has been added to and incorporated into many statistical methods in areas of robust statistics \cite{dwork2009differential}, nonparametric density estimation \cite{wasserman2010}, hypothesis testing \cite{uhlerop2013privacy,gaboardi2016differentially}, finite-sample confidence intervals \cite{karwa2017confidence}, functional data analysis \cite{hall2013differential}, network data analysis \cite{karwa2016inference}, and linear regression \cite{lei2016differentially,wang2018revisiting}.


In this paper, we provide the first differentially private multiple testing procedure. The problem of multiple testing arises in many privacy-sensitive applications such as a GWAS, where a large number of single-nucleotide polymorphisms (SNPs) are tested simultaneously for an association with a disease and the hope is to control some error rate for the significant SNPs. Perhaps the most popular error rate is the \textit{false discovery rate} (FDR), which, roughly speaking, is the expected fraction of erroneously rejected hypotheses among all rejected hypotheses. This notion of type I error rate was introduced in the seminal work of Benjamini and Hochberg \cite{BenjaminiH95}, along with the Benjamini--Hochberg procedure (BHq) that controls the FDR under certain conditions. This procedure is detailed in Algorithm~\ref{algo:originalbh}. 

Our interest in privacy-preserving FDR control arose as a group of researchers showed how to use one-way marginals, specifically, allele frequency
statistics, together with the DNA of a target individual and allele frequency statistics for the general population, to determine the target's presence or absence in the study~\cite{homer2008resolving}. In response, the US National Institutes of Health and the Wellcome Trust changed the access policy to statistics of this type in the studies they fund. Although differential privacy has been shown to permit nontrivial estimates of very large numbers of statistical queries~\cite{smallDB,hardt2010multiplicative}, the errors introduced in these techniques are -- and {\em must be}~\cite{BunUV18,dwork2015robust} -- too large for the (typical) setting, where the number of alleles exceeds the square of the number of data subjects. 

Our procedure, which is referred to as PrivateBHq henceforth (Algorithm \ref{algo:pbh2}), is derived by recognizing the iterative nature of the BHq procedure and making each iteration differentially private. PrivateBHq provides \textit{unconditional} end-to-end privacy. For now, regarding $p$-values as functions of a dataset, our proof of the privacy guarantees of PrivateBHq relies on a new definition of sensitivity that is tailor-made for $p$-values (Definition~\ref{def:multiplicative}). Loosely speaking, this definition evaluates how \textit{insensitive} $p$-values are to perturbations of any individual record in the database. All computations satisfy the definition for {\em some} choice of the privacy parameters, but not all choices of these parameters yield useful results when we enforce privacy. Popular examples of $p$-values are described in terms of these privacy parameters in Section~\ref{sec:privfdr}.


Another contribution of this paper lies in our proof of the FDR control of PrivateBHq and beyond. In short, all existing proof strategies for FDR control are invalid for PrivateBHq. Thus, a new technique for proving FDR control is needed. To this end, we
\begin{enumerate}
\item Develop a novel proof of FDR control for a class of multiple testing procedures, including the original (non-private) BHq and many of its variants -- a proof requiring different assumptions than those found in the vast literature on this topic (see Section~\ref{sec:robust}) -- and
\item
Relate the FDR control and power properties of PrivateBHq to the corresponding properties of the non-private version. 
\end{enumerate}


The outline of the remainder of the paper is as follows. The next two subsections elucidate the two contributions, namely developing PrivateBHq and proving FDR control for a class of procedures, and the following subsection consolidates privacy and inferential properties together for PrivateBHq. To make this paper self-contained, in Section \ref{sec:privfdr} we give a brief introduction to differential privacy, followed by the complete development of the PrivateBHq procedure. Section \ref{sec:robust} is devoted to establishing FDR control of a broad class of multiple testing procedures and, as an application, Section \ref{sec:appl-priv} proves FDR control of PrivateBHq and argues its power as well. The paper is concluded by a discussion in Section \ref{sec:discussion}.


\begin{algorithm}[ht]
\caption{BHq (Step-Up) Procedure}\label{algo:originalbh}
\begin{algorithmic}[1]
\REQUIRE nominal level $0 < q < 1$ and $p$-values $p_1, \ldots, p_m$
\ENSURE a set of rejected hypotheses
\STATE sort the $p$-values in increasing order: $p_{(1)}\leq p_{(2)}\leq \cdots\leq p_{(m)}$
\FOR{$j=m$ to $1$}
\IF {$p_{(j)} >  qj/m$}
\STATE continue
\ELSE
\STATE reject $p_{(1)}, \ldots, p_{(j)}$ and halt
\ENDIF
\ENDFOR\\
\it In words, the BHq (step-up) procedure finds the largest $j^\star$ such that $p_{(j^\star)} \le q j^\star/m$ and rejects all $p$-values below $q j^\star/m$.
\end{algorithmic}
\end{algorithm}

\subsection{Making BHq private}
\label{sec:making-bhq-private}

The original BHq is our starting point in developing the PrivateBHq procedure. The original procedure is non-private because the data of a single individual can affect the $p$-values of all hypotheses simultaneously, possibly changing the outcome of the BHq procedure dramatically. 

To make the BHq private, for now we need two facts about differential
privacy: (1) differential privacy is closed under composition,
permitting us to bound the cumulative privacy loss over multiple
differentially private computations.  This allows us to build complex
differentially private algorithms from simple differentially private
primitives, and (2) we will make use of the well-known \RNMa{}
(respectively, \RNMi ) primitive \cite{DworkR14}, in which
appropriately distributed fresh random noise is added to the result of
each computation, and the index of the computation yielding the
maximum (respectively, minimum) noisy value is returned. By returning
only one index the procedure allows us to pay an accuracy price for a
single computation rather than all computations.


A natural approach to obtaining a private version of BHq is by repeated use of \RNMa{}: Starting with $j=m$ and decreasing: use \RNMa{} to find the (approximately) largest $p$-value; estimate that $p$-value and, if the estimate is above a certain more conservative critical value than $qj/m$, accept the corresponding null hypothesis, remove it from consideration, and repeat. Once a
hypothesis is found with its $p$-value below the threshold, reject all the remaining hypotheses. The principal difficulty with this approach is that every iteration of the algorithm incurs a privacy loss, which can be mitigated only by increasing the magnitude of the noise used by \RNMa{}. Since each iteration corresponds to the failure of rejecting a null hypothesis, this step-up procedure is paying in privacy precisely for all null hypotheses accepted, which are by definition not the ``interesting'' ones. Moreover, recognizing that most null hypotheses in a typical GWAS would be accepted, it is fundamentally difficult to preserve information content while protecting individual privacy by emulating the step-up procedure.

Instead of starting with the largest $p$-value and considering the values in decreasing order, another approach is to start with the smallest $p$-value and consider the values in increasing order, rejecting hypotheses one by one until we find a $p$-value above some threshold. This widely studied variant is called the BHq step-down procedure, which, in contrast to the aforementioned BHq step-up procedure, finds the largest $j$ such that $p_{(i)} \le qi/m$ for all $i \le j$ and then rejects $p_{(1)}, \ldots, p_{(j)}$. Their definitions reveal that the step-down procedure shall be more conservative than its step-up counterpart. This variant, however, can assume less stringent critical values than the BHq critical values while still offering FDR control, often allowing more discoveries than the step-up counterpart \cite{stepdown}. 



If we make the natural modifications to the step-down procedure using \RNMi, also known as the \textit{Private Min} (Algorithm \ref{algo:rnm}),  instead of \RNMa, then we pay a privacy cost only for nulls rejected in favor of the corresponding alternative hypotheses, which by definition are the ``interesting'' ones.  Since the driving application of \BHq{} is to select promising directions for future investigation that have a decent chance of panning out, we can view its outcome as advice for allocating resources. Thus, a procedure that finds a relatively small number of high-quality hypotheses, still achieving FDR control, may be as useful as a procedure that finds a much larger set.

\subsection{A new technique for proving FDR control}
\label{sec:proving-fdr-control}

While various techniques have been developed in the literature for proving FDR control, they are not applicable to privacy-preserving procedures. Any privacy-preserving procedure is necessarily randomized. Consequently, the $j$th most significant noisy $p$-value may not necessarily correspond to the $j$th most significant true $p$-value. Even worse, PrivateBHq may compare a noisy $p$-value to a critical value with a different rank and, as an inevitable result, a larger $p$-value may be rejected while a smaller $p$-value is accepted. This is in stark contrast to the (non-private) BHq and most of its variants, which reject $p$-values that are contiguous in sorted order.

These facts about the PrivateBHq procedure destroy some crucial properties for proving FDR control in existing approaches. For example, it is not clear how to adapt the elegant martingale technique for FDR control, proposed by Storey, Taylor, and Siegmund \cite{storey2004strong}. In essence, this approach is to construct an empirical process indexed by a threshold under which a $p$-value is rejected. In the case of PrivateBHq, unfortunately, no such threshold exists for singling out $p$-values for declaring significance. 
Another technique that appears frequently in the FDR control literature (see, for example, \cite{prds,sarkar2008methods,finner2009false,
  roquain2011exact,bogdan2015slope,heesen2015}) is based on a crucial property of BHq: provided that a $p$-value is rejected, the effective threshold for declaring significance is completely determined by the remaining $p$-values. Unfortunately, this property is not satisfied by PrivateBHq either.

To pursue a new strategy for PrivateBHq, we observe that, although PrivateBHq might skip some of the minimum $p$-values, nevertheless it preserves a key property with high probability: if $R$ rejections are made, the largest rejected $p$-value is roughly upper bounded by $qR/m$. This motivates us to give the following definition.
\begin{definition}\label{def:compliant}
Given any cutoffs $0 < q_1 \le q_2 \le \cdots \le q_m$, a multiple testing procedure is said to be {\em compliant} with $\{q_j\}_{j=1}^m$, if all rejected $p$-values are always bounded above by $q_R$, where $R$ is the number of rejections.
\end{definition}

In the case of no rejections ($R = 0$), as a convention, the (non-existent) rejected $p$-value is considered to be bounded above by $q_R$. Compliance is an instance of a more general condition termed {\em self-consistency} \cite{blanchard2008two,blanchard2009adaptive}, which, roughly speaking, requires that any rejected $p$-value be upper bounded by a general function of the total number of rejections. Interestingly, the compliance condition as a special instance of self-consistency has not been considered in the literature. Here, we prefer to use the compliance condition as self-consistency further allows a procedure to incorporate prior information about each hypothesis into the cutoffs, which is beyond the scope of this paper. Using the BHq critical values $\{qj/m\}$ as the cutoffs (referred to as BHq-compliance henceforth), however, our condition is sufficiently general to cover many classical multiple testing procedures, including both the step-down and step-up procedures, the
generalized step-up-step-down procedures \cite{tamhane1998,sarkarstepwise} and particularly the PrivateBHq procedure (Proposition~\ref{prop:compliant}). The compliance condition is solely determined by the number of rejected $p$-values and the size of the largest one,
without requiring that each rejected $p$-value be below its associated critical value. As a consequence, this condition permits skipping the smallest $p$-values and this is well-suited for differentially private procedures. 


As revealed by this work, FDR control, roughly speaking, is a consequence of BHq-compliance together with the \textit{independence with the null} condition (Definition \ref{def:iws}). As such, our finding offers more than expected, applying to far more examples than PrivateBHq. In detail, we consider a generalized FDR \cite{kfdr,kfdrfollow} defined as
\[
\fdr_k := \E \left[ \frac{V}{R}; V \ge k \right],
\]
where $V$ denotes the number of true null hypotheses that are falsely rejected (false discoveries). The present paper primarily focuses on the case of $k \ge 2$ and, whenever clear from the context, the term FDR control in this paper stands for $\fdr_k$ control. Note that $\fdr_k$ reduces to the usual FDR if the positive integer $k$ is set to 1. This slightly relaxed FDR permits no more than $k-1$ false discoveries without any penalty, trading off for more power improvement while still maintaining a meaningful interpretation of the rejected hypotheses. The difference between the original FDR and $\fdr_k$ becomes negligible if the number of discoveries $R$ is large. As an aside, we remark that the compliance condition is not satisfied by the $\fdr_k$-controlling procedures developed in \cite{kfdr,kfdrfollow}.

Now we introduce the independence with the null condition, which is concerned with the distribution of the $p$-values. This condition is satisfied by the three examples in Appendix B. 
\begin{definition}\label{def:iws}
A set of $m$ test statistics are said to satisfy a condition referred to as \textit{independence within the null}, or IWN for short, if the true null test statistics are jointly independent.
\end{definition}
More elaboration on this new condition is carried out following Theorem~\ref{thm:optimal} below.

With the two preparatory definitions in place, we offer the following theorem. Let $m_0$ denote the number of true null hypotheses and $\truenull_0 := m_0/m$ be the true null proportion.

\begin{theorem}\label{thm:robustfdr}
If the test statistics obey the IWN condition, then any procedure that is compliant with the \BHq critical values $\{qj/m\}_{j=1}^m$ must satisfy
\begin{equation}\label{eq:fdr_intro_main}
\fdr_k \leq C_k \truenull_0 q
\end{equation}
for every $k \ge 2$, where $C_k$ is a universal constant.
\end{theorem}

We immediately obtain the following corollary.
\begin{corollary}\label{cor: bhq_app}
If the test statistics obey the IWN condition, both the BHq step-up and step-down procedures satisfy \eqref{eq:fdr_intro_main} for $k \ge 2$.
\end{corollary}

This bound involves an additional factor $C_k$, compared with the usual bound $\truenull_0 q$ in the FDR literature. Explicitly, letting $\{\xi_j\}_{j=1}^{\infty}$ be \iid~exponential random variables with mean 1, the constant is given as
\begin{equation}\label{eq:c_k_intro}
C_k = \E \left[ \max_{j \ge k} \frac{j}{\xi_1 + \cdots + \xi_j} \right].
\end{equation}
For example, $C_2 \approx 2.41, C_3 \approx 1.85, C_{10} \approx 1.32$, and $C_k$ tends to 1 as $k \goto \infty$. In particular, $C_1$ defined in \eqref{eq:c_k_intro} is infinite, and this is exactly why Theorem \ref{thm:robustfdr} does not apply to the usual FDR.  

Theorem \ref{thm:robustfdr} is optimal for all $k \ge 2$ as we show next.
\begin{theorem}\label{thm:optimal}
Given any $C < C_k$, if $q$ is sufficiently small and $m$ is sufficiently large, then there exists a \BHq-compliant procedure applied to a set of IWN $p$-values such that
\begin{equation}\nonumber
\fdr_k > C q.
\end{equation}
\end{theorem}

In the literature, existing FDR-controlling procedures often assume independence between the true null and false null test statistics (see \cite{BenjaminiH95,benjamini2006adaptive}) or certain sophisticated correlation structures between these two sets of test statistics, such as the positive regression dependent on subset (PRDS) property \cite{prds,lehmann1966some,sarkar1969some} (see also \cite{sarkarstepwise,blanchard2009adaptive}). Roughly speaking, the PRDS property holds if the test statistics exhibit certain positive dependence on each true null test statistic. In particular, the dependence between true and false nulls cannot be arbitrary. For the sake of completeness, we emphasize that the literature has considered a few cases for FDR control with an arbitrary correlation between the two sets of test statistics \cite{prds,blanchard2008two}, but, unfortunately, the associated procedures are often extremely conservative. As a well-known example, Benjamini and Yekutieli showed in Theorem 1.3 of \cite{prds} that the BHq procedure gives FDR control using critical values at level $q/(1 + \frac1{2} + \cdots + \frac1{m}) \approx q/(\log m + 0.577)$ in place of $q$. In fact, BHq with this log-factor correction could be even more conservative than the Bonferroni method \cite{mcdaneld2014genomewide}.

In contrast, Theorem \ref{thm:robustfdr} makes no assumptions regarding the dependence between the true nulls and false nulls while still controlling the FDR up to a small multiplicative factor, as the IWN condition is concerned only with the true nulls. As such, Theorem \ref{thm:robustfdr} is a contribution of independent interest to the vast FDR literature. Notably, the dependence can even be ``adversarial'' in the sense that the false null $p$-values can even be constructed as arbitrary functions of the true null $p$-values. This provides positive evidence toward understanding the robustness of the BHq procedure observed in a wide range of theoretical and empirical studies
\cite{storey2003positive,ge2008some,clarke2009robustness}.

\section{The PrivateBHq Procedure}
\label{sec:privfdr}

In this section, we first introduce the differential privacy machinery at a minimal level and then focus on developing the PrivateBHq procedure.

\subsection{Preliminaries on differential privacy}
A database $D = (d_1, d_2, \ldots, d_n) \in \mathcal{X}^n$ consists of $n$ data items (for example, health records of $n$ individuals), where $\mathcal{X}$ is a sample universe. 
Data items need not be independent (for example, health records of siblings).
Two databases $\dbax,\dbby = (d_1', d_2', \ldots, d'_n)$ are said to be {\em neighbors}, or {\em adjacent}, if they differ only in one data item. That is, there is exactly one $j$ such that $d_j \ne d'_j$. A (randomized) mechanism $\mathcal{M}$ is an algorithm that takes a database as input and releases some (randomized) response of interest.  We denote by $\range (\mathcal{M})$ the collection of all possible outputs of the mechanism $\mathcal{M}$. In the context of genome-wide association studies, a database $D$ records genotypes of individuals, and $\mathcal{M}$, for example, is a mechanism that releases the minor allele frequency of a SNP plus some random noise.

Differential privacy, now sometimes called {\em pure differential privacy}, was defined and first constructed in~\cite{DworkMNS06}.  The relaxation defined next is sometimes referred to as {\em approximate differential privacy}.

\begin{definition}[Differential Privacy \cite{DworkMNS06,DworkKMMN06}] 
\label{def:dp}
A (randomized) mechanism $\mathcal M$ is $(\e,\delta)$-differentially private for some nonnegative $\epsilon, \delta$ if
for all adjacent databases $\dbax,\dbby$ and for any measurable event $S \subset \range(\mathcal M)$,
\[
\P(\mathcal{M}(D) \in S) \le \ep{\e} \P(\mathcal M (D') \in S) + \delta.
\]
\end{definition}

Pure differential privacy is the special case where $\delta=0$. In the definition above, both databases $D, D'$ are fixed and the probabilities are taken over the randomness of the mechanism $\mathcal{M}$. The parameters $\epsilon$ and $\delta$ measure the desired privacy protection. With small $\epsilon$ and $\delta$, this definition states that the likelihood of the released response is indifferent to changing a single individual in the database, thus leaking little indication of whether a particular individual is in the database even if all the other individuals are known. This provides strong privacy protection for each individual in and outside the database.

To report a statistic $f = f(D)$ in a differentially private manner, it is necessary to randomize the mechanism. As its name suggests, the \textit{Laplace mechanism} preserves privacy by perturbing $f$ with noise generated from the Laplace distribution $\lap(\lambda)$, whose probability density is $\exp(-|x|/\lambda)/(2\lambda)$. The scale $\lambda > 0$ should be calibrated to the \textit{sensitivity} of the statistic $f$, defined as follows.

\begin{definition}\label{def:sensitivity}
Let $f$ be a real or vector valued function that takes as input a database.  The sensitivity of $f$, denoted as $\Delta f$, is the supremum of $\|f(\dba)-f(\dbb)\|_1$ over all adjacent $\dba,\dbb$, where $\|\cdot\|_1$ denotes the $\ell_1$ norm.
\end{definition}

Formally, for any function $f$ that maps databases to $\R^r$ for some positive integer $r$, we have the following result.
\begin{lemma}[Laplace Mechanism~\cite{DworkMNS06}]\label{lm:comp}
The Laplace mechanism $\mathcal{M}_L$ that outputs
\[
\mathcal{M}_L(D; f) = f(D) + (Z_1, \ldots, Z_r)
\]
preserves $(\eps,0)$-differential privacy, where $Z_j$ are \iid~draws from $\lap(\Delta f/\eps)$.
\end{lemma}

Intuitively, sensitivity quantifies the effect of any individual in the dataset on the outcome of the analysis. In this mechanism, Laplace noise with magnitude proportional to the sensitivity has the effect of masking the characteristics of any individual, thereby preserving privacy.


A simple algorithm that integrates the Laplace mechanism is the
Private Min, which is better known as the Report Noisy Min in the
literature \cite{DworkR14} and will be the building block of
PrivateBHq, introduced in Section \ref{sec:procedure}. Consider a
collection of scalar functions $f_1, \dots, f_m$. The Private Min adds
Laplace noise to each $f_j$ and then reports the smallest noisy count
(with fresh noise added) and its index. A formal description of
Private Min is given in Algorithm \ref{algo:rnm}. The following lemma
concerns its privacy property. 

\begin{lemma}\label{lm:pm_dp}
The Private Min, as detailed in Algorithm \ref{algo:rnm}, is $(\epsilon, 0)$-differentially private.
\end{lemma}

A peek at the proof of this well known lemma, which for completeness appears in the appendix, reveals that reporting each of $j^\star$ and $f_{j^\star}(D) + Z$ is $(\epsilon/2, 0)$-differentially private, hence leading to a total privacy loss of $(\epsilon/2, 0) + (\epsilon/2, 0) = (\epsilon, 0)$. Here we have used the simple fact that differential privacy loss adds up under the composition of sequential mechanisms, that is, the union of the outputs of a sequence of mechanisms that each preserve  $(\epsilon_j, \delta_j)$-differential privacy is $(\sum \epsilon_j, \sum \delta_j)$-differentially private \cite{DworkMNS06}. As an aside, the Advanced Composition Theorem \cite{DworkRV10} (see Lemma \ref{lm:advancecomp} in Section \ref{sec:preserving-privacy}), provides a much tighter bound on this privacy degradation.

\begin{algorithm}[ht]
\caption{Private Min (Report Noisy Min)} 
\label{algo:rnm}
\begin{algorithmic}[1]
\REQUIRE database $D$, functions $f_1,\dots,f_m$ each with sensitivity at most $\Delta$, and privacy parameter $\epsilon$
\ENSURE index $j^\star$ and approximation to $f_{j^\star}(D)$ 
\FOR {$j=1$ to $m$}
  \STATE set $\widetilde f_j = f_j(D) + Z_j$, where $Z_j$ is independently sampled from $\lap(2\Delta/\e)$
\ENDFOR
\STATE return $j^\star = \underset{j}{\argmin} ~ \widetilde f_j$ and $f_{j^\star}(D) + Z$, where $Z$ is a fresh draw from $\lap(2\Delta/\eps)$
\end{algorithmic}
\end{algorithm}

Looking ahead, and omitting some technicalities, PrivateBHq will operate on differentially private approximations to the logarithms of $p$-values, returned by multiple invocations of Private Min.  Since differential privacy is closed under post-processing~\cite{DworkR14}, any subsequent computation on these differentially privately obtained values can never increase privacy loss.  Thus, PrivateBHq is indeed differentially private, for all $p$-value functions. Its statistical properties will depend on the kinds of $p$-value computations that are performed, which we turn to next.

\subsection{Multiplicative sensitivity of $p$-values}
\label{sec:prelim:sensitivity}

Multiple testing procedures ubiquitously act on a set of $p$-values that are computed by functions that operate on databases. A $p$-value in our context is frequently referred to as the function on databases for computing the $p$-value instead of its numerical value, in contrast with the vast statistical literature that often does not distinguish between the function that maps a database to a $p$-value and the result of the mapping.

We now consider making $p$-value computations private as the first step toward developing a private multiple testing procedure. In many important $p$-value computations (see Example \ref{ex:binomial}), a larger $p$-value is affected more in magnitude by the change of a single data item than a smaller $p$-value. As a result, directly adding noise to the $p$-values may overprotect privacy and completely overwhelm signals in small $p$-values. This would inevitably lead to significant detection power loss as the smallest $p$-values are more likely to correspond to promising hypotheses.

Our solution will be to (very carefully) work with the logarithms of the $p$-values. This strategy is motivated by the observation that, although the (additive) sensitivity of a $p$-value may vary greatly, oftentimes the
relative change (that is, the ratio) of a $p$-value on two neighboring databases is very stable, regardless of the magnitude of the $p$-value, unless it is extremely small. In light of this observation, the sensitivity of a $p$-value, that is, the worst-case change due to the replacement of an individual in the database, is best measured multiplicatively. Below, $\eta$ and $\nu$ are nonnegative.

\begin{definition}[Multiplicative Sensitivity]
\label{def:multiplicative}
A $p$-value function $p$ is said to be $(\eta,\nu)$-{\em multiplicatively sensitive}, or $(\eta,\nu)$-{\em sensitive} for short, if for all adjacent databases $D$ and $D'$, either both $p(D), p(D') \le \nu$ or 
\[
\ep{-\eta} p(D) \le p(D') \le \ep{\eta} p(D).
\]
\end{definition}

Our PrivateBHq algorithm will make explicit use of both parameters in ensuring privacy. The parameter $\nu$ is introduced in recognition of the fact that a very small $p$-value may jump or fall by a relatively large multiplicative factor between adjacent databases. This parameter is normally much less than the Bonferroni level $q/m$~(see, for example, \cite{dudoit2008multiple}), resulting in essentially no power loss for truncating $p$-values at $\nu$. A $p$-value can satisfy different pairs of $(\eta, \nu)$-multiplicative sensitivities. In short, the two parameters $\eta$ and $\nu$ exhibit a certain trade-off relationship in the sense that one can increase (resp.~decrease) $\eta$ and decrease (resp.~increase) $\nu$ in a careful way such that a $p$-value still satisfies this condition. Every $p$-value satisfies $(\eta,\nu)$-sensitivity for {\em some} values of the parameters.  Moreover, given $p$-value functions $p_1,p_2$ with multiplicative sensitivities
$(\eta_1,\nu_1)$ and $(\eta_2,\nu_2)$ respectively, it is immediate that both functions satisfy $(\max\{\eta_1,\eta_2\}, \max\{\nu_1,\nu_2\})$-sensitivity, so given a collection of $p$-values there always exist $\eta,\nu$ so that all of the $p$-values in the collection are $(\eta,\nu)$-sensitive.

Given an $(\eta,\nu)$-sensitivity $p$-value function $p$ and a database $D$, we work with the logarithmic mapping
\[
\pi(D; p, \nu) = \log \max \{ \nu, p(D) \}
\]
This statistic satisfies $\pi(D) - \eta \le \pi(D') \le \pi(D) + \eta$ for all neighboring databases $D, D'$. In other words, $\pi$ has an {\em additive} sensitivity bounded by $\eta$. Hence, Lemma \ref{def:sensitivity} ensures that adding Laplace noise $\lap(\eta/\eps)$ to $\pi(D)$ preserves $(\epsilon, 0)$-differential privacy.

We will see below via examples that two large and important classes of $p$-value computations are $(\eta,\nu)$-sensitive for some small $\eta$ and $\nu$, with rigorous proofs given in Appendix A; as a consequence of this, preserving privacy for these $p$-values only requires a small amount of noise, leading to negligible accuracy loss. Recall that $m$ denotes the total number of hypotheses.

\begin{example}[Binomial Distribution] \label{ex:binomial}
Suppose the $n$ individuals in $D$ are, respectively, associated with $n$ \iid~Bernoulli variables $\xi_1, \ldots, \xi_n$, each of which takes the value 1 with probability $\alpha$ and the value 0 otherwise. Let $T$ denote the sum. A $p$-value $p(D)$ for testing $H_0: \alpha \le \frac12$ against the alternative $H_1: \alpha > \frac12$ is defined as
\[
p(D) = \sum_{i=t}^n \frac1{2^n} {n\choose{i}},
\]
where $t$ is the realization of $T$ on the database $D$. Denote by $t'$ the counterpart of $t$ on a neighboring database $D'$. Without loss of generality, assume $t' = t + 1$. The difference between the two $p$-values, $|p(D) - p'(D)| = \frac1{2^n} {n\choose{t}}$, attains its maximum at $t = \lfloor n/2 \rfloor$ or $\lfloor (n+1)/2 \rfloor$ ($\lfloor x \rfloor$ denotes the greatest integer that is less than or equal to $x$) and decays rapidly as $t$ deviates from $n/2$. This implies that additive sensitivity is not a good measure of the variability of this $p$-value construction.

Instead, we fix a (very) small $\nu$ and denote by $\eta$ the maximum of $\log \frac{p(D)}{p(D')}$ subject to the constraint $p(D') \ge \nu$. The $p$-value by definition is $(\eta, \nu)$-sensitive. To evaluate $\eta$, observe that the log-likelihood ratio
\[
\log \frac{p(D)}{p(D')} = \log \frac{\sum_{i=t}^n \frac1{2^n} {n\choose{i}}}{\sum_{i=t+1}^n \frac1{2^n} {n\choose{i}}} = \log \left[ 1 + \frac{ {n\choose t}}{\sum_{i=t+1}^n  {n\choose{i}}} \right] \le \frac{ {n\choose t}}{\sum_{i=t+1}^n  {n\choose{i}}}.
\]
In the appendix, it is shown that ${n\choose t}/\sum_{i=t+1}^n{n\choose{i}} \lesssim \sqrt{\frac{\log n}{n}}$ under the constraint $p(D') \ge m^{-1-c}$ for any small constant $c > 0$ if $m \le \mathtt{poly}(n)$ (that is, $m$ grows at most polynomially in $n$) as $n \goto \infty$. Therefore, we can set $\nu = m^{-1-c}$ and $\eta \asymp \sqrt{\frac{\log n}{n}}$. Note that this choice of $\nu$ is much below the Bonferroni level $q/m$.

\end{example}

\begin{example}[Truncated Exponential Distribution]\label{ex:exponential}
Let $\zeta_1, \ldots, \zeta_n$ be \iid~random variables sampled from
the density $\frac{\lambda \ep{-\lambda x}}{ 1 - \ep{-A\lambda}} \cdot
\bm{1}(0 \le x \le A)$ for positive $A$ and $\lambda$, an exponential
distribution truncated at $A$. Denote by $T = \zeta_1 + \cdots +
\zeta_n$ the sum ($T$ is a sufficient statistic for $\lambda$). To
test $H_0: \lambda = 1$ against the alternative $H_1: \lambda > 1$, we
consider the $p$-value $p(D) = \P_{\lambda=1}(T \ge t)$, where $t$ is the realization of $T$ (note that the value $t$ differs at most by $A$ between adjacent databases). With the same notations as in Example \ref{ex:binomial}, this $p$-value is $(\eta, \nu)$-multiplicatively sensitive with $\nu = m^{-1-c}$ and $\eta \asymp \sqrt{\frac{\log n}{n}}$ for any small constant $c > 0$. Similarly, the analysis applies to the case of a Gaussian distribution. In short, consider \iid~random variables $\xi_1, \ldots, \xi_n$ drawn from the normal distribution $\mathcal{N}(\mu, 1)$ truncated at $-A$ and $A$, which has density $\ep{-(x-\mu)^2/2}/\int_{-A}^A \ep{-(u-\mu)^2/2} \mathrm{d}u$. Writing $T = \xi_1 + \cdots + \xi_n$, we use the $p$-value $p(D) = \P_{\mu=0}(|T| \ge |t|)$ to test $H_0: \mu = 0$ against $H_1: \mu \ne 0$ ($t$ is the realization of $T$). Using the same proof strategy as for the exponential distribution, one can show that this $p$-value strategy is $(\eta,
\nu)$-multiplicatively sensitive with some $\nu = m^{-1-c}$ and $\eta \asymp \sqrt{\frac{\log n}{n}}$.

\end{example}

We remark that $(\eta,\nu)$-sensitivity is a worst-case guarantee on the sensitivity of a $p$-value function. Only the interpretation of the $p$-value requires the i.i.d.\,assumption. Regarding the above-mentioned two examples, the asymptotic expressions of the privacy parameter $\eta$ can be easily made precise.

As seen in both examples, the parameter $\eta$ vanishes roughly at the
rate $O(n^{-1/2})$, implying that less noise is required for privacy
protection as the sample size becomes larger. This appealing feature
is impossible without the restriction $p \ge \nu$ for some appropriate
choice of $\nu$. Specifically, in the absence of this constraint, or
equivalently by setting $\nu = 0$, we shall have $\eta= n+1$ in the
first example and $\eta = \infty$ in the second, requiring a vast or
even an infinite amount of multiplicative noise for preserving
privacy. This would completely dilute any signal of interest. To be
complete, we note that not all $p$-value computations necessarily lead
to vanishing $\eta$ and $\nu$ as $n \goto \infty$. An example from
\cite{uhlerop2013privacy,yu2014scalable} considers a privacy-preserving release of $\chi^2$-statistics computed from allelic contingency tables. For the sake of
  simplicity, here we consider $2 \times 2$ contingency tables with
  $n/2$ cases and $n/2$ controls:

  \begin{table}[!htp]
\centering
    \begin{tabular}{c c c}
\multicolumn{1}{c}{}  & \multicolumn{2}{c}{allele type} \\
 \cline{2-3}
 & major & minor\\
\hline
case & $a$ & $\frac{n}{2} - a$\\
\hline
control & $a$ & $\frac{n}{2} - a$\\
\hline
    \end{tabular}
\hspace{0.1\linewidth}
\centering
    \begin{tabular}{c c c}
\multicolumn{1}{c}{}  & \multicolumn{2}{c}{allele type} \\
 \cline{2-3}
 & major & minor\\
\hline
case & $a+1$ & $\frac{n}{2} - a - 1$\\
\hline
control & $a$ & $\frac{n}{2} - a$\\
\hline
    \end{tabular}
\hfill\\[1em]
    \caption{Two neighboring allelic contingency tables.}\label{tab:cont}
  \end{table}
In the case of a fixed $a > 5$, one can show that the two $p$-values computed from the two tables neither differ by a negligible factor nor both tend to zero as $n \goto \infty$. This fact is elaborated in detail in the appendix.

\subsection{Developing PrivateBHq}
\label{sec:procedure}

The PrivateBHq procedure (Algorithm \ref{algo:pbh2}) is the sequential composition of Algorithm~\ref{algo:peel}, which we refer to as the peeling mechanism, denoted as \peel. In a little more detail, given (non-private) $p$-value functions $p_1, \ldots, p_m$ and a prescribed number of invocations $m' \le m$, PrivateBHq first applies Private Min $m'$ times to the logarithms of the $p$-values, ``peeling off'' and removing from further consideration the approximately smallest element with each new invocation of Private Min. These $m'$ pre-selected hypotheses are thought of as promising hypotheses. In particular, the number $m'$ as an upper bound on the total number of discoveries shall be much less than $m$. This recognizes that, in many application scenarios, much fewer are truly significant in an ocean of mediocre hypotheses. 

During the peeling procedure, in order to keep track of indices within the original set, \peel removes a function from further consideration by redefining it to be $+\infty$, ensuring that it will not be returned by future invocations of the Private Min. The Laplace noise scale $\lambda$ shall be chosen to adjust for the privacy protection target, factoring in the multiplicative sensitivities of $p_1, \ldots, p_m$ and the number of invocations $m'$.

\begin{algorithm}[ht]
\caption{Peeling Mechanism \peel}
\label{algo:peel}
\begin{algorithmic}[1]
\REQUIRE database $D$, functions $f_1,\dots,f_m$, number of invocations $m'$, and Laplace noise scale $\lambda$
\ENSURE indices $i_1,\dots,i_{m'}$ and approximations to $f_{i_1}(D), \ldots, f_{i_{m'}}(D)$
\FOR {$j = 1$ to $m'$}
  \STATE let  $(i_j, \tilde f_{i_j}(D)) $ be returned by Private Min applied to $(D,f_1,\dots,f_m)$ with Laplace noise scale $\lambda$
  \STATE set $f_{i_j} \equiv +\infty$ 
\ENDFOR
\STATE return the $m'$-tuple $\{(i_1, \tilde f_{i_1}(D)),\dots, (i_{m'},\tilde f_{i_{m'}}(D) )\}$
\end{algorithmic}
\end{algorithm}

With $m'$ hypotheses yielded by \peel in place, PrivateBHq supplies quantities in logarithmic scale instead of, in the conventional setting, the $m'$ raw $p$-values and critical values to the (step-up) BHq procedure. This difference however does not affect the way BHq proceeds. To be concrete, BHq first orders the noisy values $\tilde \pi_{i_1},\ldots, \tilde \pi_{i_{m'}}$ as $\tilde \pi_{(i_1)} \le \cdots \le \tilde \pi_{(i_{m'})}$, and then rejects any corresponding hypotheses if $\tilde\pi_{i_j}$ is below $\max\{\gamma_j: \tilde\pi_{(i_j)} \le \gamma_j\}$, with the convention that $\max\emptyset = -\infty$. As we will see in Section \ref{sec:appl-priv}, the cutoffs $\gamma_1, \ldots, \gamma_{m'}$ are chosen specifically to ensure FDR control of PrivateBHq; roughly speaking, $\gamma_j$ is slightly below the logarithm of the corresponding BHq critical value $qj/m$, where the gap between the two accounts for the multiplicative sensitivity of the $p$-values and the uncertainty brought by the Laplace mechanism.

\begin{algorithm}[ht]
\caption{The \PBH~Procedure}\label{algo:pbh2}
\begin{algorithmic}[1]
\REQUIRE database $D$, parameters $\epsilon,\delta,\eta,\nu$, $(\eta, \nu)$-multiplicatively sensitive 
$p$-value functions $p_1, \dots, p_m$, number of invocations $m'$, Laplace noise scale $\lambda = \lambda(\epsilon,\delta,\eta,m')$ , and cutoffs $\gamma_1 < \cdots < \gamma_{m'}$
\ENSURE a set of up to $m'$ rejected hypotheses
\STATE set $\pi_j = \log \max\{\nu, p_j(D)\}$ for $1\leq j \leq m$
\STATE obtain $(i_1, \tilde \pi_{i_1}), \ldots, (i_{m'},\tilde \pi_{i_{m'}})$ by applying \peel to $\pi_1,\ldots, \pi_m$ with noise scale $\lambda$
\STATE apply (step-up) BHq to $\tilde \pi_{i_1},\ldots, \tilde \pi_{i_{m'}}$ with cutoffs $\gamma_1, \ldots, \gamma_{m'}$ 
\STATE return the indices of rejected hypotheses
\end{algorithmic}
\end{algorithm}

\subsection{Preserving privacy}
\label{sec:preserving-privacy}

The proof that PrivateBHq is differentially private relies on the fact that the algorithm only accesses the data through the values returned by \peel.
Thus, intuitively, the final results reported by BHq shall release no more privacy than the intermediate results yielded by \peel. This intuition is indeed true, that is, differential privacy is closed under {\em post-processing}, as shown by the following lemma.

\begin{lemma}[\cite{DworkKMMN06, wasserman2010}]\label{lm:compo}
Let $\mathcal M$ be an $(\epsilon, \delta)$-differentially private mechanism and $g$ be any (measurable) function. Then $g(\mathcal M)$ also preserves $(\epsilon, \delta)$-differential privacy.
\end{lemma}

This lemma implicitly assumes the range of the mechanism $\mathcal{M}$ falls into the domain of $g$. In our context, taking $g$ to be 
step-up BHq, Lemma \ref{lm:compo} shows that it suffices to establish the differential privacy property of \peel. By construction, each $\pi_j$ has sensitivity no more than $\eta$. Lemma \ref{lm:pm_dp} then immediately ensures that the Private Min, which is invoked sequentially $m'$ times in PrivateBHq, guarantees on its own $(2\eta/\lambda, 0)$-differential privacy. Making use of the fact that, at worst ``$(\epsilon, \delta)$'s add up'' (see the discussion right below Lemma \ref{lm:pm_dp}), one can conclude that the peeling mechanism is $(2m' \eta/\lambda, 0)$-differentially private. Equivalently, to achieve $(\epsilon,0)$-differential privacy for \peel, and therefore also
for PrivateBHq, we can set the Laplace noise scale to be $\lambda = 2m'\eta/\epsilon$. In this way, the noise level grows linearly with $m'$. 

Surprisingly, we can trade a little bit of $\delta$ for a significant improvement on $\epsilon$, as shown by the lemma below.

\begin{lemma}[Advanced Composition~\cite{DworkRV10}]\label{lm:advancecomp}
For all $\eps, \delta \ge 0$ and $\delta' > 0$, running $l$ mechanisms sequentially that are each $(\eps, \delta)$-differentially private preserves $( \eps\sqrt{2l \log (1/\delta')} +  l\eps(\ep{\eps} -1), l\delta + \delta')$-differential privacy.
\end{lemma}

This lemma holds no matter how each mechanism adaptively depends on information released by prior mechanisms. Taking $\delta = 0$ in Lemma \ref{lm:advancecomp}, we easily obtain the main theorem of this section, with its proof deferred to the appendix. This theorem shows adding Laplace noise with scale of order roughly $O(\sqrt{m'})$ is sufficient for protecting privacy of PrivateBHq.
\begin{theorem}\label{thm:main_pri}
Let $\eta,\nu$ be chosen so that all the $p$-value functions input to PrivateBHq are $(\eta, \nu)$-sensitive. Given $\epsilon \le 0.5, \delta \le 0.1$ and $m' \ge 10$, PrivateBHq with Laplace noise scale $\lambda = \eta\sqrt{10m'\log(1/\delta)}/\e$, or larger, is $(\e,\delta)$-differentially private.
\end{theorem}

We remark that the constraints on $\epsilon, \delta$, and $m'$ are used to optimize the constants for practical use. 


\section{Proving FDR Control Using a Submartingale}
\label{sec:robust}

The main purpose of this section is to prove Theorem \ref{thm:robustfdr}. The proof strategy contains two novel elements: an upper bound on $\fdr_k$ involving only true null $p$-values (Equation \eqref{eq:fdp_main} below) and a backward submartingale that allows us to use a martingale maximal inequality. In addition, this section attempts to obtain the optimal constant $C_k$ for Theorem \ref{thm:robustfdr} in Section \ref{sec:sharpening-constants}, where we give some intuition behind Theorem \ref{thm:optimal}, and considers a new variant of the FDR in Section \ref{sec:fdr-contr-large}.

Throughout the section, we focus on an arbitrary BHq-compliant procedure. That is, any $p$-value rejected by the procedure is not greater than $qR/m$, where $R$ denotes the total number of rejections.

\subsection{Controlling $\fdr_k$}
\label{sec:contr-fdr-fdr_2}

In this subsection, we prove Theorem \ref{thm:robustfdr}. However, the proof presented here does not seek to optimize the constant $C_k$ in Theorem \ref{thm:robustfdr}. We consider
\[
\fdp_k := \frac{V \mathbf{1}_{V \ge k}}{R},
\]
which gives $\fdr_k \equiv \E \fdp_k$ by taking expectation. The following upper bound on the $\fdp_k$ for $k \ge 2$ of the \BH-compliant procedure serves as the basis for our analysis:
\begin{equation}\label{eq:fdp_main}
\fdp_k \le \max_{k \le j \le m_0} \frac{q j}{m p^0_{(j)}}.
\end{equation}
Above, $p^0_{(1)} \le p^0_{(2)} \le \cdots \le p^0_{(m_0)}$ are the order statistics of the $m_0$ true null $p$-values. To prove \eqref{eq:fdp_main}, denote by $V$ the number of false rejections. If $V \le k - 1$, \eqref{eq:fdp_main} holds since $\fdp_k = 0$. Otherwise, the largest rejected true null $p$-value is at least $p^0_{(V)}$ and, therefore, one must have $p^0_{(V)} \le q R/m$ due to the compliance condition. As a consequence, we get
\begin{equation}\label{eq:fdp_k_bound}
\fdp_k  = \frac{V}{R} \le \frac{V}{m p_{(V)}^0/q} \le \max_{k \le j \le m_0} \frac{q j}{m p^0_{(j)}}.
\end{equation}
The IWN condition imposed in Theorem \ref{thm:robustfdr} ensures the joint independence of the true null $p$-values, each of which is, by definition, stochastically larger than or equal to $U(0, 1)$. Thus, the ordered true null $p$-values can be replaced by the order statistics $U_{(1)} \le U_{(2)} \le \cdots \le U_{(m_0)}$ of $m_0$ \iid~uniform random variables on $(0, 1)$, while \eqref{eq:fdp_k_bound} remains true in the expectation sense (recall that $\truenull_0 = m_0/m$):
\begin{equation}\nonumber
\fdr_k \le \E \left[ \max_{k \le j \le m_0}\frac{q j}{m U_{(j)}} \right] = q \truenull_0\E\left[ \max_{k \le j \le m_0}\frac{j}{m_0 U_{(j)}} \right].
\end{equation}
Therefore, Theorem \ref{thm:robustfdr} follows from the lemma below.
\begin{lemma}\label{lm:fdr_2_term}
Let $U_{(1)} \le \cdots \le U_{(n)}$ denote the order statistics of $n$ \iid~uniform variables on $(0, 1)$. There exists an absolute constant $c_k$ such that
\begin{equation}\nonumber
\sup_{n \ge k}\E \left[ \max_{k \le j \le n} \frac{j}{n U_{(j)}} \right] \le c_k
\end{equation}
for $k \ge 2$.
\end{lemma}

The proof of this lemma  starts by recognizing a well-known representation in law for uniform order statistics:
\begin{equation}\label{eq:u_t_define}
(U_{(1)}, \ldots, U_{(n)}) \overset{d}{=} \left( \frac{T_1}{T_{n+1}}, \ldots, \frac{T_n}{T_{n+1}} \right),
\end{equation}
where $T_j = \xi_1 + \cdots + \xi_j$ and $\xi_1, \ldots, \xi_{n+1}$ are \iid~exponential random
variables with mean 1. Writing
\[
W_j = \frac{jT_{n+1}}{T_j},
\]
Lemma \ref{lm:fdr_2_term} is equivalent to showing
\begin{equation}\label{eq:fdr_eqv}
\E \left[ \max_{k \le j \le n} \frac{W_j}{n} \right] \le c_k.
\end{equation}
Intuitively, the maximum is likely to be attained at some small index $j$ as $W_j/n$ is close to 1 for a large value of $j$, due to the law of large numbers. This intuition can be indeed made rigorous by the fact that $W_1, \ldots, W_{n+1}$ is a backward submartingale, as shown by the following lemma.

\begin{lemma}\label{lm:submartingale}
With respect to the filtration $\mathcal{F}_j := \sigma(T_j,
T_{j+1},\ldots,T_{n+1})$ for $j=1, \ldots, n+1$, the stochastic
process $W_1, \ldots, W_{n+1}$ is a backward submartingale. That is, $\E(W_j|\mathcal{F}_{j+1}) \geq W_{j+1}$ for $j=1, \ldots, n$.
\end{lemma}

The proof of Lemma \ref{lm:submartingale} is deferred to the appendix. Next, we apply this lemma to prove \eqref{eq:fdr_eqv} (hence Lemma \ref{lm:fdr_2_term} follows immediately) using the following martingale maximal inequality (for a
proof, see pages 71--73 of \cite{neveu1975discrete}).

\begin{lemma}[$\ell_1$ Martingale Maximal Inequality]\label{lm:maximal}
Let $X_1, \ldots, X_n$ be a (forward) submartingale. Then,
\[
\E\left( \max_{1 \le j \le n} X_j \right) \le \frac{\ep{}}{\ep{} - 1} \left[ 1 + \E\left( X_n \log X_n; X_n \ge 1 \right)  \right].
\]
\end{lemma}

\begin{proof}[Proof of Lemma \ref{lm:fdr_2_term}]
Since Lemma \ref{lm:submartingale} asserts that $W_j/n$ is a backward submartingale, Lemma \ref{lm:maximal} concludes
\begin{align*}
\E \left( \max_{k \le j \le n} \frac{W_j}{n} \right) &\le \frac{\ep{}}{\ep{} - 1} \left[1 + \E\left(\frac{W_k}{n}\log \frac{W_k}{n}; \frac{W_k}{n} \geq 1 \right) \right] \\
&= \frac{\ep{}}{\ep{} - 1} \left[1 + \E\left(\frac{k}{nU_{(k)}}\log \frac{k}{nU_{(k)}}; \frac{k}{nU_{(k)}} \geq 1\right) \right].
\end{align*}
To complete the proof, it suffices to show that for a fixed $k$ the expectation above involving $k/(n U_{(k)})$ is uniformly bounded for all $n \ge k$. To this end, observe that $U_{(k)}$ is distributed as $\mathrm{Beta}(k, n + 1 -k)$, and this allows us to evaluate the expectation as
\[
\begin{aligned}
\E\left(\frac{k}{nU_{(k)}}\log \frac{k}{nU_{(k)}}; \frac{k}{nU_{(k)}} \geq 1\right) &= \int_0^{\frac{k}{n}}\frac{x^{k-1}(1-x)^{n-k}}{\mathrm{B}(k, n+1-k)}\frac{k}{n x}\log\frac{k}{n x} \, \d x \\
&\le \int_0^{\frac{k}{n}}\frac{x^{k-1}}{\mathrm{B}(k, n+1-k)}\frac{k}{n x}\log\frac{k}{n x} \, \d x \\
&= \frac1{n^k \mathrm{B}(k, n+1-k)}\int_0^k k y^{k-2} \log\frac{k}{y} \, \d y \\
&= \frac1{n^k \mathrm{B}(k, n+1-k)} \cdot \frac{k^k}{(k-1)^2}.
\end{aligned}
\]
To obtain an upper bound that is independent of $n$, it suffices to show that $n^k \mathrm{B}(k, n+1-k)$ has a lower bound depending only on $k$. Indeed, this is the case:
\[
\begin{aligned}
n^k \mathrm{B}(k, n+1-k) &= n^k \frac{\Gamma(k)\Gamma(n+1-k)}{\Gamma(n+1)}\\
&= \frac{n^k (k-1)!}{n(n-1) \cdots (n-k+1)} \ge (k-1)!.
\end{aligned}
\]

\end{proof}

\subsection{Optimizing the bounds}
\label{sec:sharpening-constants}
The constant $C_k$ in Theorem \ref{thm:robustfdr} matters from a practical perspective. This section is aimed at finding the \textit{optimal} constants for all $k \ge 2$. Compared with what has been performed in Section \ref{sec:contr-fdr-fdr_2}, this improvement is based on a delicate property about the expectation in \eqref{eq:fdr_eqv}, as detailed by the following lemma.
\begin{lemma}\label{lm:d_monotone}
Define $C_k^{(n)} = \E \left[ \max_{k \le j \le n} \frac{j}{n U_{(j)}}
\right]$ for $n \ge k \ge 2$, where $U_{(j)}$'s are the order statistics of $n$
\iid~uniform variables on $(0, 1)$. Then, $C_k^{(n)} \le C_k^{(n+1)}$.
\end{lemma}

The monotonicity in Lemma \ref{lm:d_monotone} reveals that the optimal $C_k$ in \eqref{eq:fdr_eqv} takes the form (recall that $T_j = \xi_1 + \cdots + \xi_j$ is defined in \eqref{eq:u_t_define})
\begin{equation}\label{eq:c_k_rep}
C_k := \lim_{n \goto \infty} C_{k}^{(n)} = \lim_{n \goto \infty}\E\left[\max_{k \le j \le n} \frac{j T_{n+1}}{n T_j} \right].
\end{equation}
Note that $C_k$ does not seem to admit a closed-form expression. Nevertheless, this optimal constant can be easily computed via simulations. 

While relegating the full proof of Theorem~\ref{thm:optimal} to Appendix A, here we provide a proof sketch based on the construction of a BHq-compliant procedure and a set of $p$-values satisfying the IWN condition to show the optimality of $C_k$. Explicitly, let the true null $p$-values be $m_0$ \iid~uniform variables $U_1, \ldots, U_{m_0}$ between 0 and 1, and let all the $m - m_0$ false null $p$-values be 0. Denote by $j^\star$ the index $k \le j \le m_0$ that maximizes $j/U_{(j)}$. The BHq-compliant procedure rejects the $j^\star$ smallest true null $p$-values and any $\max\{\lceil m U_{(j^\star)}/q \rceil - j^\star, 0 \}$ of the false null $p$-values ($\lceil x \rceil$ denotes the least integer that is greater than or equal to $x$), which by
construction are all 0. This procedure is compliant (self-consistent) but not \textit{nonincreasing} (a procedure is called nonincreasing if it never rejects more if some $p$-value gets larger), so the FDR-controlling results in \cite{blanchard2009adaptive} do not apply to our case. Taking $q$ sufficiently small and assuming that $m - m_0$ is sufficiently large, we get $\fdp_k \approx q j^\star/(m U_{(j^\star)})$ with high probability. Consequently, we get
\[
\fdr_k \approx \E \left[ \frac{q j^\star}{m U_{(j^\star)}} \right] = \E \left[ \max_{k \le j \le m_0} \frac{j}{m_0 U_{(j)}} \right] \truenull_0 q = C_k^{(m_0)} \truenull_0 q,
\]
which tends to $C_k q$ by taking $m_0 \goto \infty$ and $m_0/m \goto 1$.

For the moment, suppose the limit can be taken under the expectation in \eqref{eq:c_k_rep}. As such, the optimal constant for $\fdr_k$ is
\begin{equation}\label{eq:ck_expe}
C_k =\E\left[ \lim_{n \goto \infty} \max_{k \le j \le n} \frac{j T_{n+1}}{n T_j} \right] = \E \left[ \max_{k \le j < \infty} \frac{j}{T_j} \right],
\end{equation}
where the last equality results from applying the strong law of large numbers to $T_{n}/n$. Recognizing that the integrable random variable $\max_{k \le j < \infty} j/T_j$ decreases to 1 almost surely as $k$ increases to infinity, Lebesgue's dominated convergence theorem readily asserts that $C_k = 1 + o_{k}(1)$, where $o_{k}(1)$ denotes a sequence of numbers tending to 0 as $k \goto \infty$. This is formally stated in the proposition below, where we consider a sequence of multiple testing problems indexed by $l$ such that both $m_l, k_l \goto \infty$ as $l \goto \infty$.

\begin{proposition}\label{prop:k_big}
Under the assumptions of Theorem \ref{thm:robustfdr}, as $k \goto
\infty$, we have $\fdr_k \le (1 + o_k(1)) q$.
\end{proposition}

To make the derivation of the optimal $C_k$ above rigorous, we must validate \eqref{eq:ck_expe}. In fact, the Vitali convergence theorem together with the following lemma ensures that the limit $\lim_{n \goto \infty}$ and expectation $\E$ can be interchanged.

\begin{lemma}\label{lm:uniform_int}
For a fixed $k \ge 2$, the sequence of random variables
\[
\max_{k \le j \le n} \frac{j T_{n+1}}{n T_j}
\]
are uniformly integrable for $n \ge k$.
\end{lemma}

While the proof of Lemma \ref{lm:uniform_int} is deferred to the appendix, the proof of Lemma \ref{lm:d_monotone} is given below.
\begin{proof}[Proof of Lemma \ref{lm:d_monotone}]
Denote by $U_{(1)} \le \cdots \le U_{(n)} \le U_{(n+1)}$ the order statistics of $n+1$ \iid~uniform random variables on $(0, 1)$. Then, $U_{(1)}/U_{(n+1)} \le \cdots \le U_{(n)}/U_{(n+1)}$ are distributed the same as the order statistics of $n$ \iid~uniform random variables on $(0, 1)$ and, moreover, are independent of $U_{(n+1)}$. Making use of this fact, we get
\[
\begin{aligned}
C_{k}^{(n+1)} &= \E \left[ \max_{k \le j \le n+1} \frac{j}{(n+1) U_{(j)}} \right]  \\
&\ge \E \left[ \max_{k \le j \le n} \frac{j}{(n+1) U_{(j)}} \right]\\
             & = \E \left[\frac{n}{(n+1) U_{(n+1)}} \cdot \max_{k \le j \le n} \frac{j}{n U_{(j)}/U_{(n+1)}} \right]\\
             & = \E \left[\frac{n}{(n+1) U_{(n+1)}} \right] \E \left[\max_{k \le j \le n} \frac{j}{n U_{(j)}/U_{(n+1)}} \right]\\
             & = \E \left[\frac{n}{(n+1) U_{(n+1)}} \right] C_k^{(n)}.
\end{aligned}
\]
Since the density of $U_{(n+1)}$ is $(n+1) x^n$ for $0 < x < 1$, we readily see that
\[
\E \left[\frac{n}{(n+1) U_{(n+1)}}  \right]= 1.
\]
This completes the last step in certifying $C_{k}^{(n+1)} \ge C_{k}^{(n)}$.
\end{proof}

\begin{figure}[!htp]
\centering
\includegraphics[width=0.5\textwidth]{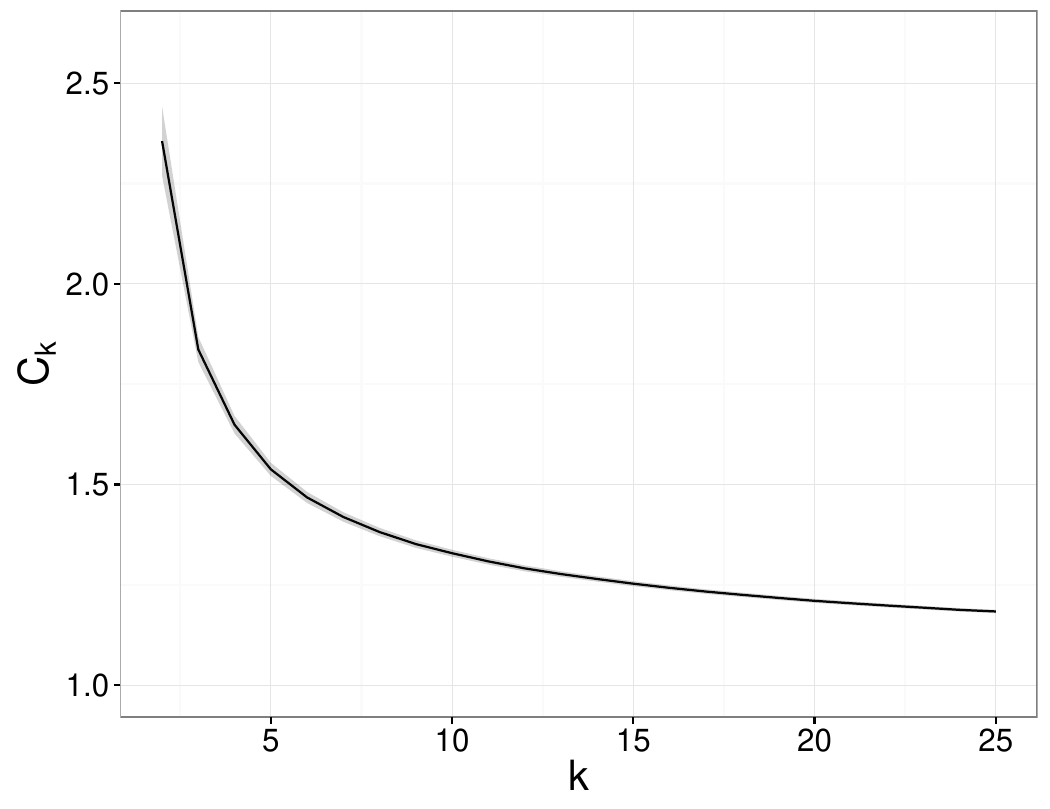}
\caption{Monte Carlo simulated values of $C_k$ using \eqref{eq:ck_expe}. The solid line indicates the maximum of $j/T_j$ over $k \le j \le 10^5$, averaged over $10^4$ runs. The (tiny) shaded band illustrates the $99\%$-coverage confidence interval for each $k$ using normal approximation.}
\label{fig:maxexp}
\end{figure}

Now, we turn to numerically evaluate $C_k$ using the expression \eqref{eq:ck_expe}. Although the distribution of each $j/T_j$ admits an analytical expression, it is however not clear how to calculate the distribution of the maximum of $j/T_j$ over $j$. In view of this difficulty, we resort to Monte Carlo simulations, and Figure \ref{fig:maxexp} presents the results that are averaged over $10^4$ independent replicates. For instance, $C_2 \approx 2.41, C_3 \approx 1.85 , C_4 \approx 1.65, C_5 \approx 1.54$, and $C_{25} \approx 1.18$. In passing, we remark that the estimated values of $C_k$ as a function of $k$ are fairly accurate as indicated by the uniformly short widths of the confidence intervals for all $k$.

\subsection{Controlling $\fdr^k$}
\label{sec:fdr-contr-large}

To further leverage the martingale-based proof idea, we consider a variant of the FDR defined as
\[
\fdr^k := \E \left[\frac{V}{R}; R \ge k \right],
\]
which includes the usual FDR as an example by taking $k = 1$. This relaxed FDR differs insignificantly from the usual FDR if a large number of discoveries are expected, which is often the case in modern multiple testing applications such as genome-wide association studies. For the moment, we do not intend to advocate the use of this new FDR definition in practice as it is clear that future investigation is needed.

In the following, we aim to prove Theorem \ref{thm:stronglyrobustfdr}, a counterpart of Theorem \ref{thm:robustfdr} for the $\fdr^k$. A similarity between the two theorems lies in that their proofs both make use of martingale arguments. That being said, the bound on the $\fdr_k$ in Theorem \ref{thm:robustfdr} cannot carry over to the $\fdr^k$ because $\fdr_k \leq \fdr^k$.

\begin{theorem}\label{thm:stronglyrobustfdr}
If the test statistics obey the IWN condition, then any BHq-compliant procedure satisfies
\begin{equation}\nonumber
\fdr^k \leq \left( 1 + \frac2{\sqrt{qk}} \right) q.
\end{equation}
for any $k \ge 1$.
\end{theorem}

A number of remarks are as follows. This theorem allows us to take $k = 1$, thus giving a bound on the usual FDR: $\fdr \le q + 2\sqrt{q}$. For example, we can set $q = 0.0024$ if the FDR is aimed to be controlled at $10\%$. Such a bound is not available in Theorem \ref{thm:robustfdr}. For completeness, the bound for $k=1$ might not sharp since Doob's $\ell^2$ martingale maximal inequality used in the proof of Theorem~\ref{thm:stronglyrobustfdr} is generally not sharp. Indeed, this bound can be improved using a careful treatment of \eqref{eq:fdp_main} (see \cite{fdrlinking}). For $k \ge 2$, the bound here is larger than that in Theorem \ref{thm:robustfdr}, namely $2/\sqrt{qk} \ge C_k - 1$, due to the optimality of $C_k$ and the fact $\fdr^k \ge \fdr_k$. The following proof actually establishes a stronger bound, $\truenull_0 q + 2\sqrt{\truenull_0 q/k}$, on the $\fdr^k$. Recall that $\truenull_0$ is the true null proportion $m_0/m$.

\begin{proof}[Proof of Theorem~\ref{thm:stronglyrobustfdr}]
Due to the compliance condition, the number of false discoveries satisfies
\[
V \le \# \left\{ i \text{ is true null}: p_i \le \frac{qR}{m} \right\}.
\]
Thus, we get an upper bound on $\fdp := \frac{V}{R}$ (with the convention $0/0 = 0$) that takes the following form:
\begin{equation}\nonumber
\fdp \le \max_{R \le j \le m}\frac{\# \{ i \text{ is true null}: p_i \le qj/m \}} {j}.
\end{equation}
Consequently, we get
\begin{equation}\label{eq:fdrk_above}
\fdp^k := \frac{V \mathbf{1}_{R \ge k}}{R} \le \max_{k \le j \le m}\frac{\# \{ i \text{ is true null}: p_i \le qj/m \}} {j}.
\end{equation}
Similar to what has been argued in Section \ref{sec:contr-fdr-fdr_2}, the inequality \eqref{eq:fdrk_above} still holds if all true null $p$-values are replaced by $m_0$ \iid~uniform variables $U_1, \ldots, U_{m_0}$ on $(0, 1)$. This observation shows that it suffices to prove
\begin{equation}\label{eq:high_power_basis}
\E\left[ \max_{k \le j \le m}\frac{\#\big{\{}1 \le i \le m_0: U_i \le qj/m\big{\}}}{j} \right] \le \left(1 + 2/\sqrt{qk} \right) q.
\end{equation}

To show \eqref{eq:high_power_basis}, denote by $V_j = \#\{1 \le i \le m_0: U_i \le qj/m\}$ and $Y_j = V_j/j$. Conditional on $Y_{j+1}$, for every $i \in \{1 \le i \le m_0: U_i \le q(j+1)/m\}$ the random variable $U_i$ is uniformly distributed on $[0, q(j+1)/m]$. Hence, the conditional expectation of $V_j$ given $Y_{j+1}$ is
\[
\E(V_j|Y_{j+1}) = \frac{V_{j+1}\frac{qj}{m}}{\frac{q(j+1)}{m}} = \frac{jV_{j+1}}{j+1},
\]
which is equivalent to
\[
\E(Y_j|Y_{j+1}) = Y_{j+1}\,.
\]
In words, $Y_j$ is a backward martingale and, as a consequence, $(Y_j - qm_0/m)_+$ is a backward submartingale. This fact allows us to apply Doob's $\ell^2$ martingale maximal inequality to $(Y_j - q m_0/m)_+$, yielding
\[
\begin{aligned}
\E \left[ \max_{k \le j \le m} \left( Y_j - \frac{q m_0}{m} \right)_+^2  \right] &\le \left( \frac{2}{2-1} \right)^2 \E \left( Y_{k} - \frac{q m_0}{m} \right)_+^2 \\
&\le 4\E\left( Y_k - \frac{q m_0}{m} \right)^2 \\
&= \frac{4q m_0 (1 - qk/m)}{k m}\\
& < \frac{4 \truenull_0 q }{k}.
\end{aligned}
\]
Using Jensen's inequality, the left-hand side of \eqref{eq:high_power_basis} satisfies
\[
\begin{aligned}
\E \left[ \max_{k \le j \le m} Y_j \right] &\le \frac{q m_0}{m} + \E \left[ \max_{k \le j \le m} \left( Y_j - \frac{q m_0}{m} \right)_+ \right]\\
&\le \truenull_0 q+ \sqrt{\E \left[ \max_{k \le j \le m} \left( Y_j - \frac{q m_0}{m} \right)_+^2 \right]}\\
&\le \truenull_0 q + 2\sqrt{\frac{\truenull_0 q}{k}}.
\end{aligned}
\]
\end{proof}


\section{FDR Control and Power of PrivateBHq}
\label{sec:appl-priv}

As an application of Theorem \ref{thm:robustfdr}, this section considers FDR control and power of PrivateBHq. Throughout this process, we take the assumptions of Theorem~\ref{thm:main_pri} as given. That is, we assume that each $p_i$ is $(\eta,\nu)$-sensitive and the parameters satisfy $\epsilon \le 0.5, \delta \le 0.1$, and $m' \ge 10$. From Theorem~\ref{thm:main_pri}, PrivateBHq preserves $(\epsilon, \delta)$-differential privacy, and for brevity this fact will not be reiterated in this section.

The proposition below demonstrates that the PrivateBHq is indeed compliant by making the cutoffs $\{\gamma_j\}$ in Algorithm \ref{algo:pbh2} slightly more stringent than the logarithms of the BHq critical values.
\begin{proposition}\label{prop:compliant}
For any $0 < q < 1$, use the cutoffs
\begin{equation}\label{eq:gamma_cut} 
\gamma_j = \log\frac{qj}{m} - \frac{\eta \sqrt{10 m'\log(1/\delta)} \log(6m'/q)}{\epsilon}
\end{equation}
for $j = 1, \ldots, m'$ in PrivateBHq. Under the assumptions of Theorem \ref{thm:main_pri}, this procedure is compliant with the BHq critical values $qj/m$ with probability at least $1 - 0.1q$.
\end{proposition}

As a remark, the first term $\log\frac{qj}{m}$ in \eqref{eq:gamma_cut} corresponds to the non-private cutoff and the second term $ - \frac{\eta \sqrt{10 m'\log(1/\delta)} \log(6m'/q)}{\epsilon}$ is used to handle the added noise. Notably, the constant $0.1$ above can be replaced by any positive constant provided that the second term is appropriately scaled. The proof of Proposition~\ref{eq:gamma_cut} is given later after Theorem~\ref{thm:pbhq_fdr}.


The compliance condition shown in Proposition \ref{prop:compliant} together with Theorem \ref{thm:robustfdr} implies FDR control of PrivateBHq. More precisely, letting $\mathcal{C}$ denote the event that the rejected $p$-values are compliant, we have
\[
\begin{aligned}
\fdr_k &= \E\left( \fdp_k; \mathcal{C} \right) + \E\left( \fdp_k; \overline{\mathcal{C}} \right)\\
&\le C_k q + \P(\overline{\mathcal{C}}) \le (C_k + 0.1)q
\end{aligned}
\]
for every $k \ge 2$. As such, to control the FDR at level, say $10\%$ (a common level used in practice), we can set $q = 0.1/(C_k + 0.1)$ in PrivateBHq. This proves the following theorem.
\begin{theorem}\label{thm:pbhq_fdr}
Under the same assumptions as in Proposition \ref{prop:compliant} and if the test statistics satisfy the IWN condition, the PrivateBHq procedure gives
\begin{equation}\nonumber
\fdr_k \le (C_k + 0.1)q
\end{equation}
for all $k \ge 2$.
\end{theorem}

To prove Proposition \ref{prop:compliant}, we first present a simple lemma that gives a concentration bound on Laplace random variables, and its proof can be found in the appendix.
\begin{lemma}\label{lem:slack}
Let $Z_1, \ldots, Z_n$ be \iid~$\lap(\lambda)$ random variables. For any $0 < \alpha < 1$, the following two statements are true:
\begin{enumerate}
\item With probability at least $1 - \alpha$, all $Z_j$ are larger than $-\lambda \log\frac{n}{2\alpha}$.
\item With probability at least $1 - \alpha$, all $|Z_j|$ are smaller than $\lambda \log\frac{n}{\alpha}$.
\end{enumerate}
\end{lemma}

\begin{proof}[Proof of Proposition \ref{prop:compliant}]
Let $\tilde\pi_{i_j} = \log\max\{\nu, p_{i_j}\} + Z_{i_j}$ be yielded by \peel in Algorithm \ref{algo:pbh2}, where $Z_{i_j}$ follows $\lap(\lambda)$ for $j = 1, \ldots, m'$. The parameter $\lambda = \eta\sqrt{10m'\log(1/\delta)}/\e$ is as in Theorem \ref{thm:main_pri}. Taking $\alpha = 0.1q$, Lemma \ref{lem:slack} shows that
\begin{equation}\label{eq:zij_small}
Z_{i_j} > -\lambda \log\frac{m'}{2\times 0.1 q} > -\frac{\eta \sqrt{10m'\log(1/\delta)} \log(6m'/q)}{\epsilon}.
\end{equation}
uniformly for $j = 1, \ldots, m'$ with probability at least $1 - 0.1q$.

Next, we show that on the event \eqref{eq:zij_small}, PrivateBHq is compliant. Denote by $R_{\textnormal{Pt}}$ the number of rejections made by this procedure. If $\tilde\pi_{i_j}$ is rejected, it must satisfy
\[
\log\max\{\nu, p_{i_j}\} + Z_{i_j} \le \gamma_{R_{\textnormal{Pt}}} = \log\frac{qR_{\textnormal{Pt}}}{m} - \frac{\eta \sqrt{10 m'\log(1/\delta)} \log(6m'/q)}{\epsilon}.
\]
Plugging \eqref{eq:zij_small} into this display gives
\[
\log\max\{\nu, p_{i_j}\} \le \log\frac{qR_{\textnormal{Pt}}}{m}.
\]
Thus, $p_{i_j} \le q R_{\textnormal{Pt}}/m$ for all rejected $p_{i_j}$ on the event \eqref{eq:zij_small}, which happens with probability at least $1 - 0.1q$. This completes the proof.

\end{proof}

Next, Theorem \ref{thm:power} shows that the PrivateBHq procedure with a slightly inflated nominal level is at least as powerful as the BHq step-down procedure. The proofs of this theorem and its corollary are deferred to the appendix. To state the theorem, let $R_{\textnormal{SD}}$ denote the number of rejections made by the (non-private) step-down procedure.

\begin{theorem}\label{thm:power}
Fix $q$ and assume $\nu \le q/m$. Under the assumptions of Theorem \ref{thm:pbhq_fdr}, run the PrivateBHq procedure at level
\[
q' = q \ep{\frac{24 \eta\sqrt{m'\log(1/\delta)} \log m}{\e}}
\]
and the BHq step-down procedure at level $q$. Then, the numbers of rejections satisfy
\begin{equation}\label{eq:r_ineq}
R_{\textnormal{Pt}} \ge \min\{R_{\textnormal{SD}}, m'\}
\end{equation}
with probability tending to one as $m \goto \infty$.
\end{theorem}

When $R_{\textnormal{SD}} \ge m'$ and the event \eqref{eq:r_ineq}
happens, PrivateBHq must reject all $p$-values passing through
\peel. In the case where non-null $p$-values are significant enough to pass through \peel, this fact suggests that PrivateBHq achieves high power. This high-power property, however, is appealing if $q'$ is only
slightly larger than $q$ or, put more simply, the number $24 \eta\sqrt{m'\log(1/\delta)} \log m/\e$ is small. With regard to Examples \ref{ex:binomial} and \ref{ex:exponential}, this is equivalent to have a sufficiently large sample size $n$. The following corollary formalizes this point.


\begin{corollary}\label{coro:examples}
In Examples \ref{ex:binomial} and \ref{ex:exponential}, fix $\epsilon, \delta$ and assume $m' \le \min\{n^{1 - c}, m\}$ for constant $c > 0$. Under the assumptions of Theorem \ref{thm:power}, the claims of both Theorems \ref{thm:pbhq_fdr} and \ref{thm:power} hold as $m, n \goto \infty$ if PrivateBHq is performed at level $(1+c')q$ for a sufficiently small constant $c' > 0$.
\end{corollary}

\subsection{Empirical evaluation}
\label{sec:empir-comp}


In this subsection, we evaluate the price paid for privacy in terms of FDR control and power in the PrivateBHq procedure. The aim is to provide a better picture of how much detection power would be compromised due to privacy guarantees for FDR control. For completeness, this comparison includes a private version of Bonferroni's method, which is referred to PrivateBonf in this paper. PrivateBonf is perhaps the simplest baseline for private multiple hypothesis testing. This procedure is detailed as follows. As in Algorithm~\ref{algo:pbh2}, let $p_1, \ldots, p_m$ be $(\eta, \nu)$-sensitive $p$-values and set $\pi_j = \log\max\{p_j, \nu\}$ for all $j$. The PrivateBonf procedure adds independent $\lap(\tilde\lambda)$ noise to all $\pi_j$, where $\tilde\lambda = \eta\sqrt{10 m \log(1/\delta)}/(2 \epsilon)$, and rejects those with noisy counts below
\[
\log\frac{q}{m} - \frac{\eta\sqrt{10 m \log(1/\delta)}\log(5m/q)}{2 \epsilon}.
\]

The following result is concerned with privacy and family-wise error rate (FWER) control of PrivateBonf. Note that the FWER denotes the probability that at least one false positive is made. The proof is deferred to the appendix. As an aside, the privacy guarantee in this result might be improved by using the sparse vector technique \cite{hardt2010multiplicative} and this is left for future investigation.
\begin{proposition}\label{prop:bonf}
Under the assumptions of Theorem~\ref{thm:main_pri}, the following two statements are true:
\begin{enumerate}
\item
PrivateBonf is $(\epsilon, \delta)$-differentially private;
\item
PrivateBonf satisfies $\mathrm{FWER} \le 1.1 q$.
\end{enumerate}
\end{proposition}

\begin{figure}[!htp]
\centering
\begin{subfigure}[b]{0.45\textwidth}
\centering
\includegraphics[width=0.8\textwidth]{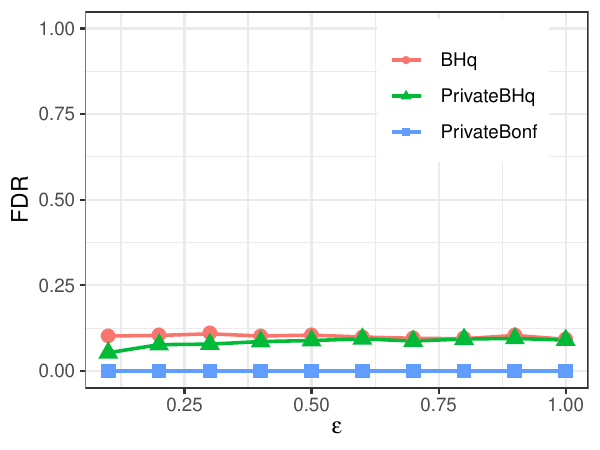}
\end{subfigure}
\begin{subfigure}[b]{0.45\textwidth}
\centering
\includegraphics[width=0.8\textwidth]{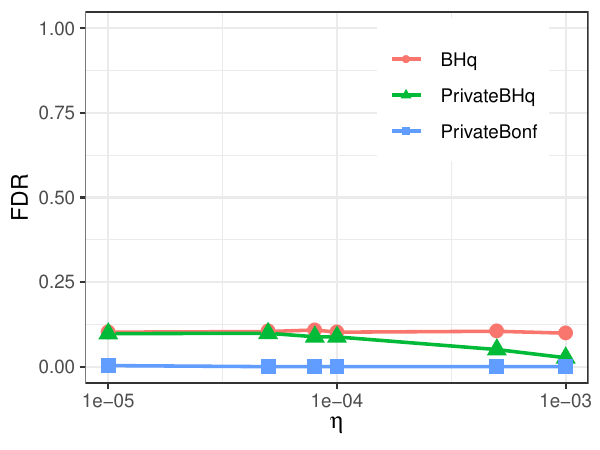}
\end{subfigure}\\
\centering
\begin{subfigure}[b]{0.45\textwidth}
\centering
\includegraphics[width=0.8\textwidth]{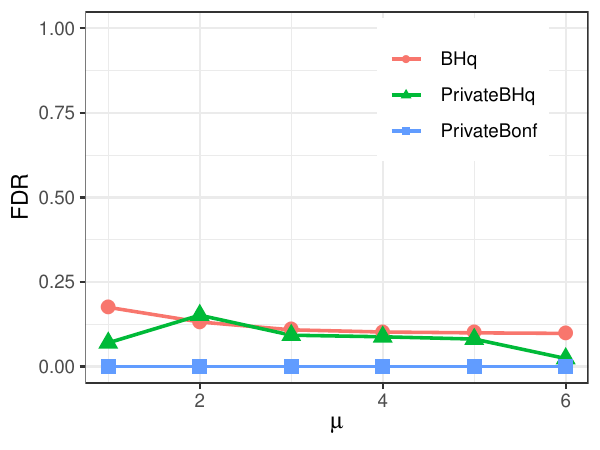}
\end{subfigure}
\begin{subfigure}[b]{0.45\textwidth}
\centering
\includegraphics[width=0.8\textwidth]{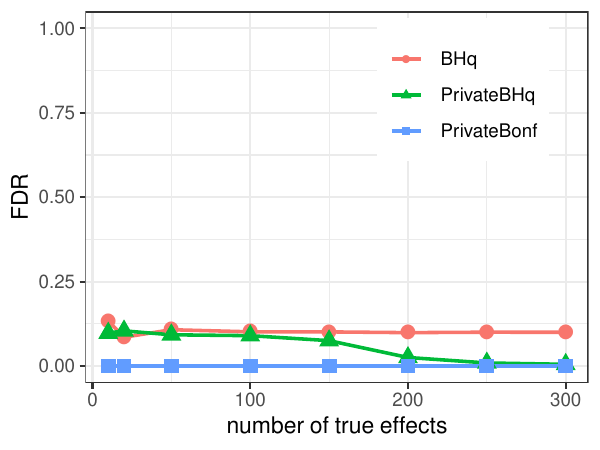}
\end{subfigure}
\hfill
\caption{The FDR of BHq, PrivateBHq, and PrivateBonf, plotted against varying $\epsilon, \eta, \mu, m_1$, respectively, and averaged over 100 independent replicates. Note that PrivateBHq in general discovers fewer than BHq and has a smaller FDR than BHq as well.}
\label{fig:fdp_eval}
\end{figure}

\begin{figure}[!htp]
\centering
\begin{subfigure}[b]{0.45\textwidth}
\centering
\includegraphics[width=0.8\textwidth]{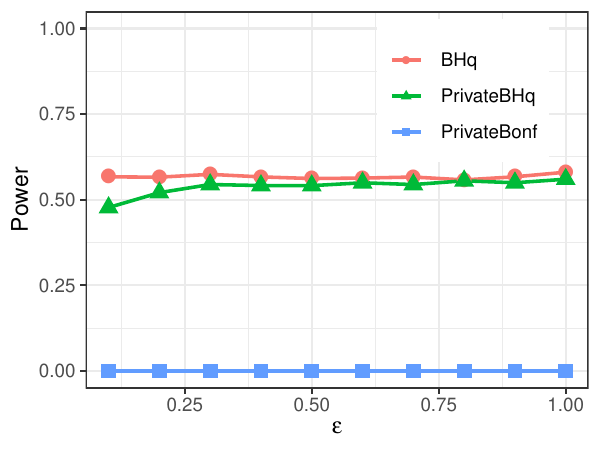}
\end{subfigure}
\begin{subfigure}[b]{0.45\textwidth}
\centering
\includegraphics[width=0.8\textwidth]{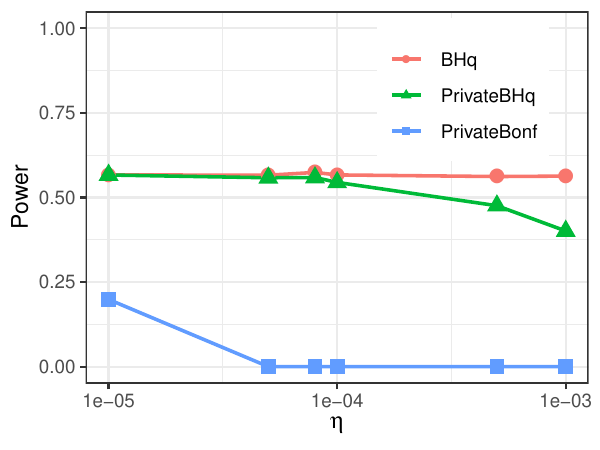}
\end{subfigure}\\
\centering
\begin{subfigure}[b]{0.45\textwidth}
\centering
\includegraphics[width=0.8\textwidth]{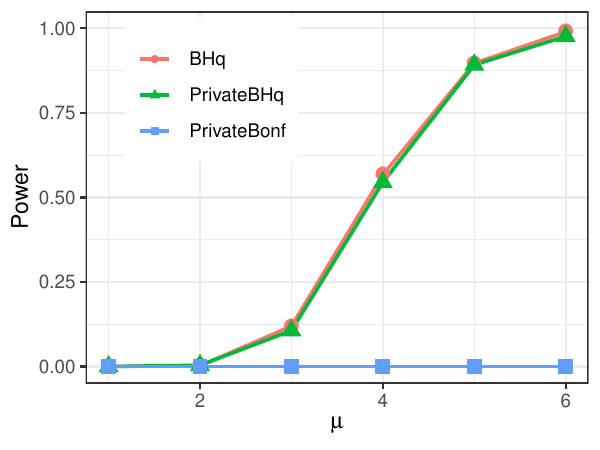}
\end{subfigure}
\begin{subfigure}[b]{0.45\textwidth}
\centering
\includegraphics[width=0.8\textwidth]{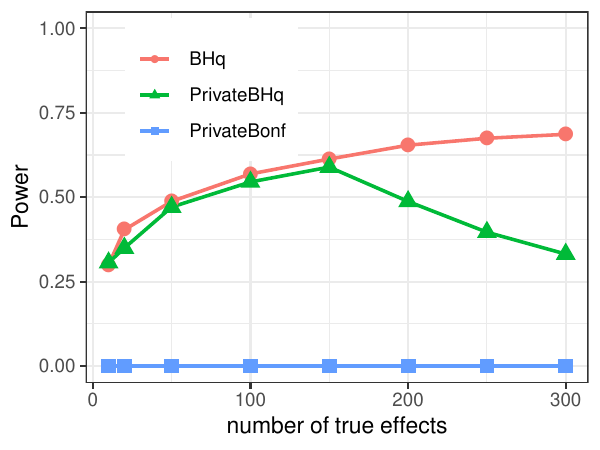}
\end{subfigure}
\hfill
\caption{The power of BHq, PrivateBHq, and PrivateBonf, plotted against varying $\epsilon, \eta, \mu, m_1$, respectively, and averaged over 100 independent replicates.}
\label{fig:power_eval}
\end{figure}

Figures~\ref{fig:fdp_eval} and \ref{fig:power_eval} present, respectively, the FDR and the power of PrivateBHq, the (non-private) BHq step-up procedure, and PrivateBonf by simulations. Unless specified, we set $m = 10^5, m' = 100, q = 0.1, \eta = 10^{-4}, \nu = 0.5q/m, \epsilon = 0.5$, and $\delta = 0.001$. To construct the $p$-values, we let $p_i = \Phi(\xi_i - \mu)$ for $i = 1, \ldots, m_1$ and $p_{m_1+1}, \ldots, p_m$ be i.i.d.~uniform variables on $(0, 1)$, where $\Phi$ is the CDF of $\mathcal{N}(0, 1)$, $\xi_1, \ldots, \xi_{m_1}$ are i.i.d.~copies of $\mathcal{N}(0, 1)$, and the default values of $\mu$ and $m_1$ (the number of true effects) are set to $4$ and $100$, respectively. In summary, the FDR of PrivateBHq is empirically controlled at $q$ in almost all scenarios, though
Theorem~\ref{thm:pbhq_fdr} is only concerned with $\fdr_k$ for $k \ge 2$. Moreover, PrivateBonf is uniformly the least powerful among the three procedures. This is not surprising given that PrivateBonf is inherently developed for FWER control. Looking closely, the performance of PrivateBHq is comparable to that of BHq when $\epsilon$ is not too small and $\eta$ is not too large. Notably, the power of PrivateBHq deteriorates when the number of true effects exceeds 150, which is due to the truncation of the PrivateBHq procedure at $m' = 100$.

For completeness, we refer interested readers to a set of numerical comparisons of an $\fdr_k$-controlling procedure \cite{kfdr,kfdrfollow}, PrivateBHq, and PrivateBonf in the appendix.


\section{Discussion}
\label{sec:discussion}

This paper has developed a privacy-preserving multiple testing procedure termed PrivateBHq for FDR control. On the privacy side, we propose a new notion of sensitivity tailored to $p$-values and recognize the sequential nature of the BHq (step-down) procedure so as to keep PrivateBHq efficient under the differential privacy constraint. Differential privacy of this whole pipeline follows from the composition nature of differential privacy. On the statistical side, as a major contribution of the paper, it is proved that a large class of multiple testing procedures, including the step-up, step-down, and PrivateBHq procedures, control the $\fdr_k$ only provided the joint independence of the true null test statistics. A novel aspect of this result lies in the absence of any assumption on the dependence between the true nulls and false nulls. Notably, some recent progress has been made along this direction using the the FDR-linking technique~\cite{fdrlinking}.

Looking forward, our work raises a number of open questions. First, it
would be interesting to take into account prior knowledge, such as the
importance of hypotheses and beliefs about which are true nulls, into
the design of a differentially private procedure. Second, it would be of interest to develop private procedures for control of other popular error rates such as the $q$-value~\cite{storey2003statistical}. Moreover, it is natural to
wonder if the bound in Theorem~\ref{thm:robustfdr} can improve by
imposing some structure on the dependence between the true null and
false null test statistics. Third, recognizing the vital importance of \peel in our PrivateBHq, an interesting direction is to investigate alternatives to \peel, such as the oneshot approach to the problem of private top-$k$ selection~\cite{qiao2021oneshot}. Last, it would be interesting to consider other notions of privacy such as concentrated differential privacy and Gaussian differential privacy \cite{dwork2016concentrated,bun2016concentrated,dong2019gaussian}.


Finally, we wish to make a connection to a remarkable property of differential privacy: it protects against false discoveries due to adaptive data analysis, where an analysis is informed by prior interactions with the same database \cite{DworkFHPRR14,dwork2015generalization,bassily2016algorithmic}. Adaptivity is ubiquitous in practice as the analyst is often not clear a priori what are the right questions to ask about a database. In the multiple testing context, this issue arises when hypotheses are adaptively selected based on prior discoveries. A question of great interest is to develop a multiple testing procedure that continues to preserve privacy in the presence of adaptivity.

\section*{Acknowledgements}
We would like to thank Michal Linial, Daniel Rubin, Abba Krieger, Xinran Li, and Sanat Sarkar for helpful discussions and useful comments about an early version of the manuscript. We are grateful to three anonymous reviewers for their constructive comments. C.~D.~and W.~S.~were supported in part by the NSF via grant CCF-1763314.

{\small
\bibliographystyle{abbrv}
\bibliography{refs}
}


\clearpage
\appendix
\newcommand{\kl}{\operatorname{KL}}

\section{Proofs}
\label{sec:technical-proofs}

This section proves all results made without proof in the main text. Below, the proofs are listed in order of appearance of their corresponding results.

\begin{proof}[Proof of Lemma \ref{lm:pm_dp}]
For an arbitrary index $1 \le j \le m $ and a measurable set $S \subset \R$, it suffices to prove that
\[
\frac{\P(\tilde f_j ~ \text{is the smallest and}~ f_j(D) + Z \in S)}{\P(\tilde f'_j ~ \text{is the smallest and}~ f_j(D') + Z \in S)} \le \ep{\epsilon},
\]
where $\tilde f'_j$ is the counterpart of $\tilde f_j$ evaluated on an adjacent database $D'$. This inequality is equivalent to
\begin{equation}\label{eq:pm_proof}
\frac{\P(\tilde f_j ~ \text{is smallest}) \P(f_j(D) + Z \in S | \tilde f_j ~ \text{is smallest})}{\P(\tilde f'_j ~ \text{is smallest}) \P(f_j(D') + Z \in S | \tilde f'_j ~ \text{is smallest})} \le \ep{\e}.
\end{equation}
First, releasing the index of the smallest noisy count is $(\epsilon/2, 0)$-differentially private, which has been proven by Claim 3.9 in Section 3.3 of \cite{DworkR14}. That is,
\[
 \frac{\P(\tilde f_j ~ \text{is smallest})}{\P(\tilde f'_j ~ \text{is smallest})} \le \ep{\epsilon/2}.
\]
Second, observe that by assumption $|f_j(D) - f_j(D')| \le \Delta$. Then Lemma \ref{lm:comp} shows that
\[
\frac{\P(f_j(D) + Z \in S | \tilde f_j ~ \text{is smallest})}{\P(f_j(D') + Z \in S | \tilde f'_j ~ \text{is smallest})} \le \ep{\epsilon/2}.
\]
Combining the last two displays concludes that \eqref{eq:pm_proof} is bounded by $\ep{\epsilon}$. This finishes the proof.

\end{proof}

\begin{proof}[Proof of Example \ref{ex:binomial}]
We use $t - 1$ in place of $t$. Under the constraint that $p(D), p(D') \ge \nu$, we aim to prove that 
\[
\frac{ {n\choose t - 1}}{\sum_{i=t}^n {n\choose{i}}} \le \eta
\]
if $\eta \asymp \sqrt{(\log n) /n}$. Without loss of generality, assume $t \ge n/2$, where a well-known result is
\begin{equation}\label{eq:kl_wk}
\sum_{i=t}^n \frac1{2^n} {n\choose{i}} \le \ep{- n \kl(\frac{t}{n}, \frac12)}.
\end{equation}
Above, the Kullback--Leibler divergence is defined as
\[
\kl(a, b) = a\log\frac{a}{b} + (1 - a) \log\frac{1 - a}{1 - b}.
\]
It is easy to show that
\[
\kl\left(a, \frac12 \right) \ge 2\left(a - \frac12\right)^2.
\]
Therefore, plugging
\[
\sum_{i=t}^n \frac1{2^n} {n\choose{i}} \ge \nu = m^{-1-c} = \frac1{\mathtt{poly}(n)}
\]
into \eqref{eq:kl_wk}, we get
\[
\frac{t}{n} \le \frac12 + O\left(\sqrt{\frac{\log n}{n}} \right)
\]
or, put differently,
\[
t \le \frac{n}{2} + O\left(\sqrt{n\log n} \right).
\]
Therefore, we can assume $t \le 7n/8$. Provided that $n/2 \le t \le 7n/8$, we can apply Littlewood's theorem \cite{littlewood1969probability,mckay1989littlewood}. Letting $u = (2t - n)/\sqrt{n}, \rho = 1 - t/n$ and $\Xi(x) = \Phi(-x)/\phi(x)$, where $\Phi(x)$ and $\phi(x)$ are the cumulative distribution function and density function of $\mathcal{N}(0,1)$ respectively, this theorem gives
\[
\sum_{i=t}^n \frac1{2^n} {n \choose i} = \left( 1 + O(1/n) \right)\Phi(-u) \ep{A_1 + A_2/\sqrt{\rho(1 - \rho)n}},
\]
where
\[
A_1 = \frac{u^2}{2} - \left(t - \frac12 \right) \log \frac{2t}{n} - \left(n - t + \frac12 \right) \log \frac{2(n - t)}{n}
\]
and
\[
A_2 = \frac{1 - 2\rho}{6}\left[\frac{1 - u^2}{\Xi(u)} + u^3  \right] + \frac{1/\Xi(u) - u}{2}.
\]
Because $t \le n/2 + O(\sqrt{n \log n})$ as $n \goto \infty$, we have $u = O(\sqrt{\log n}), \rho = \frac12 - o(1)$. Making use of the fact that $\Xi(u) = (1 + o(1))/u$, we see that
\begin{equation}\label{eq:bi_tail}
\begin{aligned}
\sum_{i=t}^n \frac1{2^n} {n \choose i} &= \left( 1 + O(1/n) \right)\Phi(-u) \ep{A_1} (1 + o(1))\\
&= \left( 1 + o(1) \right) \frac{\Phi(-u)}{\sqrt{2\pi} \phi(u)} \ep{- (t - \frac12) \log \frac{2t}{n} - (n - t + \frac12) \log \frac{2(n - t)}{n}}\\
&= \left( 1 + o(1) \right) \frac{\Phi(-u)}{\sqrt{2\pi} \phi(u)} \ep{- t\log \frac{2t}{n} - (n - t) \log \frac{2(n - t)}{n}}\\
&= \left( 1 + o(1) \right) \frac{\Phi(-u)}{\sqrt{2\pi} \phi(u)} \ep{- n \kl(\frac{t}{n}, \frac12)}\\
&= \left( 1 + o(1) \right) \cdot \left( 1 + o(1) \right) \frac1{\sqrt{2\pi} u} \ep{- n \kl(\frac{t}{n}, \frac12)}\\
& \ge \frac{O(1)}{\sqrt{2\pi \log n}} \ep{- n \kl(\frac{t}{n}, \frac12)}.
\end{aligned}
\end{equation}

Next, we consider
\[
\frac1{2^n} {n \choose t-1}
\]
By Stirling's formula and using the fact that $t = (0.5 + o(1))n$, we get
\[
\begin{aligned}
\frac1{2^n} {n \choose t-1} &= \frac{t}{n+1 - t}\frac1{2^n} {n \choose t} \\
                                         &= (1+ o(1)) \frac1{2^n} {n \choose t} \\
                                         &= (1+ o(1))  \sqrt{\frac2{\pi n}} \cdot \frac{n^n}{2^n t^{t} (n - t)^{n - t}}\\
                                         &= (1+ o(1))  \sqrt{\frac2{\pi n}} \cdot \ep{- n \kl(\frac{t}{n}, \frac12)}
\end{aligned}
\]
Thus, we get
\[
\begin{aligned}
\frac{\frac1{2^n} {n \choose t-1}}{\sum_{i=t}^n \frac1{2^n} {n \choose i}} &\le O(1) \frac{(1+ o(1))  \sqrt{\frac2{\pi n}} \cdot \ep{- n \kl(\frac{t}{n}, \frac12)}}{\frac1{\sqrt{2\pi \log n}} \ep{- n \kl(\frac{t}{n}, \frac12)}}\\
& = O\left(\sqrt{\frac{\log n}{n}}\right).
\end{aligned}
\]
Thus, with $\nu = m^{-1-c}$, we can choose $\eta \asymp \sqrt{\frac{\log n}{n}}$.

\end{proof}

\begin{proof}[Proof of Example \ref{ex:exponential}]
Let $\zeta, \zeta_1, \ldots, \zeta_n$ be \iid~exponential variable with $\lambda = 1$ truncated at $A$. Consider the cumulant generating function
\[
\kappa(\theta) = \log \E \ep{\theta \zeta}.
\]
As in the proof of Example \ref{ex:binomial}, it does not lose generality by assuming $t > n \E \zeta$. Write $a = t/n > \E \zeta$ and let $\theta_{a}$ be the root of the saddle-point equation
\begin{equation}\label{eq:saddle}
\kappa'(\theta_a) = a.
\end{equation}
In particular, $\E_{\theta_a} \zeta = a$. Note that under $\E_{\theta}$ the density of $\zeta$ is
\[
\frac{\lambda \ep{-\lambda x}}{ 1 - \ep{-A\lambda}} \ep{\theta x - \kappa(\theta)} \cdot \bm{1}(0 \le x \le A).
\]
Through exponential tilting, we get
\[
\begin{aligned}
\P(T \ge na) &= \P\left( \zeta_1 + \cdots + \zeta_n \ge na \right) \\
&= \ep{-n(a \theta_a - \kappa(\theta_a))} \E_{\theta_a}\ep{-\theta_a(T - na)} \bm{1}(T \ge na).
\end{aligned}
\]
Using saddle point approximation, we get
\[
\E_{\theta_a}\ep{-\theta_a(T - na)} \bm{1}(T \ge na) = \frac{1 + o(1)}{\sqrt{2\pi \kappa''(\theta_a) n} \theta_a}
\]
Thus, we have
\begin{equation}\label{eq:na_kappa}
\P(T \ge na) = (1+o(1)) \frac{\ep{-n(a \theta_a - \kappa(\theta_a))}}{\sqrt{2\pi \kappa''(\theta_a) n} \theta_a}.
\end{equation}

Next, we evaluate $\kappa''(\theta_a)$ and $\theta_a$. Denote by $\mu$ and $\sigma^2$ the mean and variance of 
$\zeta$, respectively. We get $\theta_a = o(1)$ and $\kappa''(\theta_a) = \var_{\theta_a}(\zeta) = \sigma^2 + o(1)$. In particular, from \eqref{eq:saddle} we get
\[
\theta_a = (1 + o(1))\frac{a - \mu}{\sigma^2},
\]
which gives
\[
a \theta_a - \kappa(\theta_a) = (1 + o(1)) \frac{(a - \mu)^2}{2\sigma^2}.
\]
Plugging this display into
\[
\ep{-n(a \theta_a - \kappa(\theta_a))} \ge \nu = m^{-1 - c} = \frac1{\mathtt{poly}(n)}.
\]
gives
\[
t = n \mu + O(\sqrt{n \log n}).
\]
Therefore, we get
\[
\E_{\theta_a}\ep{-\theta_a(T - na) }\bm{1}(T \ge na) = \frac{1 + o(1)}{\sqrt{2\pi \kappa''(\theta_a) n} \theta_a} = \frac{O(1)}{\sqrt{\log n}},
\]
which together with \eqref{eq:na_kappa} yields
\begin{equation}\label{eq:tail_exp}
\P\left( T \ge na \right)  = \frac{\ep{-n(a \theta_a - \kappa(\theta_a))}}{\sqrt{\log n}}.
\end{equation}

To evaluate the ratio
\[
\frac{\P\left( na - A \le T < na \right)}{\P\left( T \ge na \right)},
\]
it remains to approximate $\P\left( na - A \le T < na \right)$. We use the local central limit theorem to do this. Explicitly, using the local central limit theorem, we get
\begin{equation}\label{eq:local_clt}
\begin{aligned}
\P\left( t - A \le T < t \right) &= \ep{-n(a \theta_a - \kappa(\theta_a))} \E_{\theta_a}\ep{-\theta_a(T - na)} \bm{1}(na - A \le T < na)\\
& \le  \ep{-n(a \theta_a - \kappa(\theta_a))} \E_{\theta_a}\ep{\theta_a A}  \bm{1}(na - A \le T < na)\\
& = \ep{\theta_a A} \ep{-n(a \theta_a - \kappa(\theta_a))} \P_{\theta_a}(na - A \le T < na)\\
& = \ep{\theta_a A} \ep{-n(a \theta_a - \kappa(\theta_a))} \left( \frac{A}{\sqrt{2\pi n} \sigma_a} + o(1/\sqrt{n}) \right)\\
& = (1 + o(1))\frac{A\ep{\theta_a A} \ep{-n(a \theta_a - \kappa(\theta_a))}}{\sqrt{2\pi n} \sigma_a},
\end{aligned}
\end{equation}
where $\sigma_a$ is the standard deviation of $\zeta$ tilted at $\theta_a$. That is, $\sigma_a = \sqrt{\var_{\theta_a}{\zeta}} = \sigma + o(1)$.

Finally, combing \eqref{eq:tail_exp} and \eqref{eq:local_clt} gives
\[
\frac{\P\left( na - A \le T < na \right)}{\P\left( T \ge na \right)} \le O(1) \sqrt{\frac{\log n}{n}}.
\]
As such, we can choose
\[
\eta = O(1) \sqrt{\frac{\log n}{n}}.
\]

\end{proof}

\begin{proof}[An example of $p$-value computation from
  \cite{uhlerop2013privacy,yu2014scalable}] 
Now we show that the contingency table example in Section~\ref{sec:prelim:sensitivity} does not give $p$-values that are $(\eta, \nu)$-sensitive with
  some $\eta, \nu \goto 0$ even if $n \goto \infty$. 
In particular, we focus on two adjacent tables as shown in Table
\ref{tab:cont}. The $\chi^2$-statistic of the left table is
\[
\chi^2_L = 0
\]
because $a \times (n/2-a) - a \times (n/2-a) = 0$ and, as a consequence,
the corresponding $p$-value is
\[
p_L \approx \P(\chi_1^2 \ge \chi^2_L) = 1.
\]
Next, for the right table the statistic equals
\[
\begin{aligned}
\chi^2_R &= \frac{\left[(a+1)(n/2-a) - a(n/2-a-1) \right]^2}{n}\\
& \times \left[\frac1{\frac{n}{2} (2a+1)} + \frac1{\frac{n}{2} (2a+1)} +
  \frac1{\frac{n}{2} (n-2a-1)} + \frac1{\frac{n}{2} (n-2a-1)}\right]\\
& = \frac{n}{4} \times \left[\frac1{\frac{n}{2} (2a+1)} + \frac1{\frac{n}{2} (2a+1)} +
  \frac1{\frac{n}{2} (n-2a-1)} + \frac1{\frac{n}{2} (n-2a-1)}\right]\\
& = \frac12 \left[\frac1{2a+1} + \frac1{2a+1} +
  \frac1{n-2a-1} + \frac1{n-2a-1}\right]\\
& = \frac1{2a+1} + \frac1{n-2a-1}.
\end{aligned}
\]
Now, assuming $5 \le a \ll n$, we get
\[
\chi_R^2 = \frac1{2a+1} + o(1),
\]
leading to
\[
p_R \approx \P(\chi_1^2 \ge \chi_R^2) \approx 2\Phi\left(-\frac1{\sqrt{2a+1}} \right).
\]

Thus, in this example both $p_R$ and $p_L$ are bounded below away from 0 and the ratio
$2\Phi\left(-\frac1{\sqrt{2a+1}} \right)$ does not tend to 1 as $n
\goto \infty$. As a consequence of this, it is impossible to have both
vanishing $\eta$ and $\nu$ for this $p$-value computation.

\end{proof}

\begin{proof}[Proof of Theorem \ref{thm:main_pri}]
The PrivateBHq procedure acts on the intermediate results
\[
(i_1, \tilde \pi_{i_1}), \ldots, (i_k,\tilde \pi_{i_k})
\] 
provided by the \peel. Hence, Lemma \ref{lm:compo} implies that it suffices to establish the $(\e, \delta)$-differential privacy for \peel as a part of PrivateBHq. By Lemma \ref{lm:pm_dp}, each Private Min in \peel is $(\frac{2\epsilon}{\sqrt{10k \log(1/\delta)}}, 0)$-differentially private. Then Lemma \ref{lm:advancecomp} immediately asserts that \peel preserves $(\tilde\epsilon, \delta)$-differential privacy, where
\begin{equation}\label{eq:pri_use_comp}
\begin{aligned}
\tilde\epsilon &= \frac{2\epsilon}{\sqrt{10k \log(1/\delta)}} \sqrt{2k\log(1/\delta)} + k \frac{2\epsilon}{\sqrt{10k \log(1/\delta)}} (\ep{\frac{2\epsilon}{\sqrt{10k \log(1/\delta)}}} - 1)\\
&= \frac{2\epsilon}{\sqrt{5}} + \frac{2\epsilon\sqrt{k}}{\sqrt{10\log(1/\delta)}} (\ep{\frac{2\epsilon}{\sqrt{10k \log(1/\delta)}}} - 1).
\end{aligned}
\end{equation}
Recognizing that $\frac{2\epsilon}{\sqrt{10k \log(1/\delta)}} \le 0.0659$ under the assumptions $\epsilon \le 0.5, \delta \le 0.1$ and $k \ge 10$, we get
\[
\begin{aligned}
\frac{2\epsilon\sqrt{k}}{\sqrt{10\log(1/\delta)}} \left[\ep{\frac{2\epsilon}{\sqrt{10k \log(1/\delta)}}} - 1\right] &\le \frac{2\epsilon\sqrt{k}}{\sqrt{10\log(1/\delta)}} \times 1.034 \times \frac{2\epsilon}{\sqrt{10k \log(1/\delta)}}\\
& \le 0.0899 \epsilon.
\end{aligned}
\]
Substituting this line into \eqref{eq:pri_use_comp} yields
$\tilde\epsilon \le 2\epsilon/\sqrt{5} + 0.0899 \epsilon \le \epsilon$, as desired.

\end{proof}

\begin{proof}[Proof of Lemma~\ref{lm:submartingale}]
The proof is similar to that of Example 5.6.1 in \cite{durrett}.  By scaling,
assume that $\xi_i$ are exponential random variables with parameter 1,
i.e, $\E \xi_i = 1$. Note that $W_j$ is measurable with respect to
$\mathcal{F}_j$. In the proof, we first consider the conditional
expectation $\E( W_j^{-1}|\mathcal{F}_{j+1})$, then return to $\E(
W_j|\mathcal{F}_{j+1})$ by applying Jensen's inequality. Specifically, we have
\[
\E \left[ \frac{\xi_{l}}{j T_{m+1}}\big{|}\mathcal{F}_{j+1} \right] = \frac{1}{j T_{m+1}} \E(\xi_l |\mathcal{F}_{j+1})
\]
because $T_{m+1}$ is measurable in $\mathcal{F}_{j+1}$. Next, observe that by symmetry we get
\[
\E ( \xi_{l}|\mathcal{F}_{j+1}) = \E ( \xi_{k}|\mathcal{F}_{j+1})
\]
for any $l, k \le j + 1$. Combining the last two displays gives
\begin{align*}
\E\left[ \frac{T_j}{j T_{m+1}}\big{|}\mathcal{F}_{j+1} \right] &= \E \left[ \frac{T_j}{(j+1) T_{m+1}} \big{|}\mathcal{F}_{j+1} \right] + \sum_{l=1}^j \E \left[ \frac{\xi_l}{j(j+1) T_{m+1}}\big{|}\mathcal{F}_{j+1} \right] \\
&= \E \left[ \frac{T_j}{(j+1) T_{m+1}}\big{|}\mathcal{F}_{j+1} \right] + \sum_{l=1}^j \E \left[ \frac{\xi_{j+1}}{j(j+1) T_{m+1}}\big{|}\mathcal{F}_{j+1} \right]\\
&= \E \left[ \frac{T_{j+1}}{(j+1) T_{m+1}}\big{|}\mathcal{F}_{j+1} \right] \\
&= \frac{T_{j+1}}{(j+1) T_{m+1}}\,.
\end{align*}
To complete the proof, note that Jensen's inequality asserts that
\[
\E( W_j|\mathcal{F}_{j+1}) \geq \frac{1}{\E( W_j^{-1}|\mathcal{F}_{j+1})} = \frac{(j+1) T_{m+1}}{T_{j+1}} = W_{j+1}\,,
\]
as desired.
\end{proof}

\begin{proof}[Proof of Theorem \ref{thm:optimal}]
In addition to the proof sketch in Section \ref{sec:sharpening-constants}, it remains to show that
\[
\frac1{q}\E \left[ \frac{j^\star}{j^\star + \max\{\lceil m U_{(j^\star)}/q \rceil - j^\star, 0 \}} \right] \goto C_k 
\]
as $q \goto 0, m \goto \infty, m - m_0 \goto \infty$ and $m_0/m \goto 1$. Since $j^\star = O_{\P}(1)$ and $m U_{(j^\star)}$ is bounded below away from 0 as $m_0 \goto \infty$ and $m_0/m \goto 1$, one can show that\footnote{One needs to ensure that $m - m_0$ is larger than $\max\{\lceil m U_{(j^\star)}/q \rceil - j^\star, 0 \}$ with high probability. Thus, $q$ should tend to $0$ slowly as $m - m_0 \goto \infty$.}
\[
\max\{\lceil m U_{(j^\star)}/q \rceil - j^\star, 0 \} = m U_{(j^\star)}/q - j^\star + O_{\P}(1).
\]
Thus, we get
\[
\begin{aligned}
\frac1{q} \cdot \frac{j^\star}{j^\star + \max\{\lceil m U_{(j^\star)}/q \rceil - j^\star, 0 \}} &= \frac{j^\star}{m U_{(j^\star)}} + o_{\P}(1)\\
& = \max_{k \le j \le m_0}\frac{j}{m U_{(j)}} + o_{\P}(1).
\end{aligned}
\]
By assumption, we have
\[
\E\left[\max_{k \le j \le m_0}\frac{j}{m U_{(j)}}\right]  = (1 + o(1))\E\left[\max_{k \le j \le m_0}\frac{j}{m_0 U_{(j)}}\right] \goto C_k
\]
as $m_0 \goto \infty$.

To complete the proof, the last step is to show that
\[
\frac1{q} \cdot \frac{j^\star}{j^\star + \max\{\lceil m U_{(j^\star)}/q \rceil - j^\star, 0 \}}
\]
is bounded by an integrable random variable. To this end, we note that if $m U_{(j^\star)}/q \ge j^\star$, then
\[
\frac1{q} \cdot \frac{j^\star}{j^\star + \max\{\lceil m U_{(j^\star)}/q \rceil - j^\star, 0 \}} \le \frac{j^\star}{m U_{(j^\star)}},
\]
and otherwise $q \ge m U_{(j^\star)}/j^\star$, yielding
\[
\frac1{q} \cdot \frac{j^\star}{j^\star + \max\{\lceil m U_{(j^\star)}/q \rceil - j^\star, 0 \}} \le \frac1{q} \cdot \frac{j^\star}{j^\star + 0} \le \frac{j^\star}{m U_{(j^\star)}}.
\]
Note that $\frac{j^\star}{m U_{(j^\star)}}$ is bounded by an integrable random variable by resorting the representation using $T_j$.
\end{proof}

\begin{proof}[Proof of Lemma \ref{lm:uniform_int}]
We first prove the case where $k \ge 3$. The uniform integrability follows if we show
\begin{equation}\label{eq:unif_int}
\sup_{n \ge k}\E \left(\max_{k \le j \le n} \frac{j T_{n+1}}{n T_j} \right)^2 < \infty.
\end{equation}
As proved by Lemma \ref{lm:submartingale}, $\frac{j T_{n+1}}{n T_j}$ is a backward submartingale for $j = k, k+1, \ldots, n$. Thus, by Doob's maximal inequality, we get
\[
\E \left(\max_{k \le j \le n} \frac{j T_{n+1}}{n T_j} \right)^2 \le 4 \E \left(\frac{k T_{n+1}}{n T_k} \right)^2
\]
So, the proof would be completed if we verify
\[
\sup_{n \ge k}\E \left(\frac{k T_{n+1}}{n T_k} \right)^2 < \infty.
\]
To this end, note that
\[
\begin{aligned}
\E \frac{k^2 T^2_{n+1}}{n^2 T^2_k} &= \frac{k^2}{n^2}\E \frac{(T_k + T_{n+1} - T_k)^2}{T^2_k}\\
&= \frac{k^2}{n^2}\E \frac{T_k^2 + 2T_k (n+1 - k) + (n+1-k)^2 + (n+1-k)}{T^2_k}\\
&\le \frac{k^2}{n^2}\E \frac{T_k^2 + 2nT_k + n^2}{T^2_k}\\
&\le \frac{k^2}{n^2} + \frac{2k^2}{n} \E \frac1{T_k} + k^2 \E \frac1{T_k^2}\\
&\le 1 + 2k\E \frac1{T_k} + k^2 \E \frac1{T_k^2},\\
\end{aligned}
\]
which is finite if $k \ge 3$. Thus, \eqref{eq:unif_int} holds for $k \ge 3$. 

Next, we turn to the case of $k = 2$. Recognize that
\[
\begin{aligned}
\max_{2 \le j \le n} \frac{j T_{n+1}}{n T_j} &\le \frac{2 T_{n+1}}{n T_2} + \max_{3 \le j \le n} \frac{j T_{n+1}}{n T_j}\\
                                                          &= \frac2n +  \frac{2 (T_{n+1} - T_2)}{n T_2} + \max_{3 \le j \le n} \frac{j T_{n+1}}{n T_j}
\end{aligned}
\]
Since $\frac2n$ and $\max_{3 \le j \le n} \frac{j T_{n+1}}{n T_j}$ are both uniformly integrable, it is sufficiently to show the uniform integrability of $\frac{2 (T_{n+1} - T_2)}{n T_2}$ for $n \ge 2$. To this end, note that
\[
\begin{aligned}
\E \left[ \frac{2 (T_{n+1} - T_2)}{n T_2} \right]^{1.5} &= \frac{2^{1.5}}{n^{1.5}} \E (T_{n+1} - T_2)^{1.5} \E T_2^{-1.5} \\
&\le \frac{2^{1.5}}{n^{1.5}} \left[ \E (T_{n+1} - T_2)^2 \right]^{\frac{3}{4}} \E T_2^{-1.5} \\
&= \frac{2^{1.5}}{n^{1.5}} \left[ (n-1)^2 + n-1 \right]^{\frac{3}{4}} \E T_2^{-1.5} \\
& <  2^{1.5} \E T_2^{-1.5},
\end{aligned}
\]
which is finite. The proof is complete.

\end{proof}

\begin{proof}[Proof of Lemma \ref{lem:slack}]
We first consider part one. Note that
\begin{equation}\label{eq:lap_tail}
\P\left( Z_j \le -\lambda\log\frac{n}{2\alpha} \right) = \frac12 \times \frac{2\alpha}{n} = \frac{\alpha}{n}.
\end{equation}
Hence, taking a union bound, we get
\[
\begin{aligned}
\P\left( \text{all } Z_j >  -\lambda\log\frac{n}{2\alpha} \right) &= 1 - \P\left( \min Z_j \le  -\lambda\log\frac{n}{2\alpha} \right)\\
& \ge 1 - \sum_{j=1}^ n \P\left( Z_j \le  -\lambda\log\frac{n}{2\alpha} \right)\\
& =  1 - n \frac{\alpha}{n}\\
& = 1 - \alpha.
\end{aligned}
\]

The proof of part two is the follows the same reasoning except using 
\begin{equation}\nonumber
\P\left( |Z_j| \ge \lambda\log\frac{n}{\alpha} \right) = \frac{\alpha}{n}
\end{equation}
in place of \eqref{eq:lap_tail}.

\end{proof}

\begin{proof}[Proof of Theorem \ref{thm:power}]
In the proof below, we replace the assumption on the nominal level with the relaxed assumption that $q \ge 6m^{-1.5}$. Let $0 < \alpha, \alpha' < 1$ be specified later. Denote by $R'_{\textnormal{SD}} = \min\{R_{\textnormal{SD}}, m'\}$ and let $p_{j_1}, \ldots, p_{j_{R'_{\textnormal{SD}}}}$ be the $R'_{\textnormal{SD}}$ smallest $p$-values. By the construction of the step-down procedure, we get
\[
\max\{p_{j_1}, \ldots, p_{j_{R'_{\textnormal{SD}}}}\} \le \frac{q R'_{\textnormal{SD}}}{m}.
\]
First, we point out that the first $R'_{\textnormal{SD}}$ selections (without added noise) in the peeling stage of PrivateBHq, denoted as $\pi_{i_1}, \ldots, \pi_{i_{R'_{\textnormal{SD}}}}$, obey
\begin{equation}\label{eq:pi_pri_max}
\max\{\pi_{i_1}, \ldots, \pi_{i_{R'_{\textnormal{SD}}}}\} \le \log\frac{qR'_{\textnormal{SD}}}{m} + 2\lambda \log\frac{m^2}{\alpha}
\end{equation}
with probability at least $1 - \alpha$. To show this, we recognize that, with probability at least $1 - \alpha$, all the $mm'$ noise terms added by PrivateBHq are bounded in absolute value by
\begin{equation}\label{eq:noise_bound}
\lambda \log\frac{mm'}{\alpha} \le \lambda \log\frac{m^2}{\alpha}
\end{equation}
by using Lemma \ref{lem:slack}. Now, consider the $l$the step of invoking the peeling, where $1 \le l \le R'_{\textnormal{SD}}$. Note that at least one of $\pi_{j_1}, \ldots, \pi_{j_{R'_{\textnormal{SD}}}}$ remains on the list. Hence, at least one candidate for \RNM~is no greater than
\[
\log\max\left\{\frac{qR'_{\textnormal{SD}}}{m}, \nu \right\} + \lambda \log\frac{m^2}{\alpha} = \log\frac{qR'_{\textnormal{SD}}}{m} + \lambda \log\frac{m^2}{\alpha}.
\]
Then, it must hold that
\[
\pi_{i_l} + Z_{i_l}' \le \log \frac{qR'_{\textnormal{SD}}}{m} + \lambda \log\frac{m^2}{\alpha},
\]
where on the event \eqref{eq:noise_bound} $Z_{i_l}' \ge -\lambda \log\frac{m^2}{\alpha}$. Therefore, we get
\[
\pi_{i_l} \le \log \frac{qR'_{\textnormal{SD}}}{m} + 2\lambda \log\frac{m^2}{\alpha},
\]
thus confirming \eqref{eq:pi_pri_max}.

With added noise, the counts satisfy
\begin{equation}\label{eq:pi_tilde_lamn}
\begin{aligned}
\tilde\pi_{i_l} &\le \log \frac{qR'_{\textnormal{SD}}}{m} + 2\lambda \log\frac{m^2}{\alpha} + \lambda \log\frac{m'}{2\alpha'}\\
              & \le -\log\frac{m}{qR'_{\textnormal{SD}}} + 2\lambda \log\frac{m^2}{\alpha} + \lambda \log\frac{m}{2\alpha'}\\
\end{aligned}
\end{equation}
with probability at least $1 - \alpha - \alpha'$ for all $l = 1, \ldots, R'_{\textnormal{SD}}$. Next, take
\begin{equation}\label{eq:take_given_m}
2\lambda \log\frac{m^2}{\alpha} + \lambda \log\frac{m}{2\alpha'} \le  \frac{16\eta\sqrt{k\log(1/\delta)} \log m}{\e}
\end{equation}
as given for the moment. Then, from \eqref{eq:pi_tilde_lamn} we get
\begin{equation}\label{eq:16_10_logm}
\tilde\pi_{i_l} \le -\log\frac{m}{qR'_{\textnormal{SD}}} +  \frac{16\eta\sqrt{k\log(1/\delta)} \log m}{\e}
\end{equation}
for all $i = 1, \ldots, R'_{\textnormal{SD}}$ with probability at least $1 - \alpha - \alpha'$. Now, we turn to verify \eqref{eq:take_given_m}, which is equivalent to
\[
2 \log\frac{m^2}{\alpha} + \log\frac{m}{2\alpha'} \le \frac{16}{\sqrt{10}}\log m.
\]
To this end, it suffices to set $\alpha = m^{-0.014}$ and $\alpha' = m^{-0.029}/2$. Since both $\alpha, \alpha' \goto 0$ as $m \goto \infty$, we see that \eqref{eq:16_10_logm} holds with probability tending to one.

Recognizing \eqref{eq:16_10_logm}, to reject all of these $R'_{\textnormal{SD}}$ hypotheses using PrivateBHq, it is sufficient to have
\[
-\log\frac{m}{qR'_{\textnormal{SD}}} +  \frac{16\eta\sqrt{k\log(1/\delta)} \log m}{\e} \le  -\log\frac{m}{q' R'_{\textnormal{SD}}} - \frac{\eta \sqrt{10k\log\frac{1}{\delta}} \log\frac{6m'}{q'}}{\epsilon},
\]
which is equivalent to
\begin{equation}\label{eq:qq_relation}
\begin{aligned}
\log\frac{q'}{q} &\ge \frac{\eta \sqrt{10k\log\frac{1}{\delta}} \log\frac{6m'}{q'}}{\epsilon} + \frac{16\eta\sqrt{k\log(1/\delta)} \log m}{\e}\\
& = \frac{\eta\sqrt{k\log(1/\delta)} }{\e} \left[\sqrt{10}\log\frac{6m'}{q'} + 16 \log m \right].
\end{aligned}
\end{equation}
Since $q \ge 6m^{-1.5}$, we get
\[
\begin{aligned}
\sqrt{10}\log\frac{6m'}{q'} &\le \sqrt{10}\log\frac{6m}{q} \\
&\le \sqrt{10}\log\frac{6m}{6 m^{-1.5}} \\
&< 8\log m.
\end{aligned}
\]
Hence, \eqref{eq:qq_relation} is implied by
\[
\log\frac{q'}{q} \ge  \frac{24\eta\sqrt{k\log(1/\delta)} \log m}{\e},
\]
which is in fact an equality by assumption. Thus, the proof is complete.

\end{proof}

\begin{proof}[Proof of Corollary \ref{coro:examples}]
A careful look at the proof of Theorem \ref{thm:pbhq_fdr} reveals that the event in Proposition \ref{prop:compliant} holds with probability at least $1 - q/12$. As such, the bound on the $\fdr_k$ in Theorem \ref{thm:pbhq_fdr} can be strengthened to $(C_k + 1/12)q$.

In light of the above, we set $c'$ such that
\[
(C_k + 1/12)(1+c') = C_k + 0.1.
\]
Then, PrivateBHq at level $(1 + c')q$ controls the $\fdr_k$ at level $(C_k+0.1)q$ as ensured by Theorem \ref{thm:pbhq_fdr}.

It remains to prove that the claim of Theorem \ref{thm:power} also holds. To this end, we only need to show that $q'$ that is given in the statement of Theorem \ref{thm:power} is less than $(1+c')q$ for sufficiently large $m,n$. That is,
\[
(1+c')q > q' \equiv q \ep{\frac{24 \eta\sqrt{m'\log(1/\delta)} \log m}{\e}}.
\]
In both examples, $\eta = n^{-0.5 + o(1)}$ and $\log m = \log \mathtt{poly}(n) = n^{o(1)}$. Thus, we have
\[
\ep{\frac{24 \eta\sqrt{m'\log(1/\delta)} \log m}{\e}} \goto 0
\]
due to $m' \le n^{1-c}$. Therefore, in words, PrivateBHq at level $(1+c')q$ should be at least as powerful as the truncated BHq step-down procedure with probability tending to one.

\end{proof}

\begin{proof}[Proof of Proposition~\ref{prop:bonf}]

We first prove that PrivateBonf is $(\epsilon, \delta)$-differentially private. To this end, we start by observing that each $\pi_j$ corrupted by $\lap(\tilde\lambda)$ noise is $\epsilon'$-differentially private, where
\[
\epsilon' = \frac{\eta}{\tilde\lambda} = \frac{2\epsilon}{\sqrt{10 m \log(1/\delta)}}.
\]
Using the Advanced Composition Theorem, therefore, it suffices to show that
\[
\epsilon'\sqrt{2m \log(1/\delta)} + m \epsilon'(\ep{\epsilon'} - 1) \le \epsilon,
\]
Under the assumptions of Theorem~\ref{thm:main_pri}, we have
\[
\ep{\epsilon'} - 1 \le 1.034\epsilon'
\]
Thus, the proof would be completed once we show
\[
\epsilon'\sqrt{2m \log(1/\delta)} + 1.034 m \epsilon'^2 \le \epsilon,
\]
which is equivalent to
\[
\frac2{\sqrt{5}} + 1.034 \times \frac{2\epsilon}{5\log(1/\delta)} \le 1.
\]
This inequality can be easily verified.

Next, we turn to show the second statement. As with the proof of Theorem~\ref{thm:pbhq_fdr}, we only need to show that, with probability at least $1 - 0.1q$, all noisy counts $\pi_j + \lap(\tilde\lambda)$ with $p_j > q/m$ are above
\[
\log\frac{q}{m} - \frac{\eta\sqrt{10 m \log(1/\delta)}\log(5m/q)}{2 \epsilon}.
\]
This statement is implied if $m$ \iid~$\lap(\tilde\lambda)$ noise terms are all above $- \frac{\eta\sqrt{10 m \log(1/\delta)}\log(5m/q)}{2 \epsilon}$ with probability at least $1 - 0.1q$. This claim is true by invoking Lemma~\ref{lem:slack}.

\end{proof}


\section{More Simulation Results}
\label{sec:more-simul-results}

\subsection{Comparions of $\fdr_k$-controlling procedures with others}
\label{sec:comp-set-numer}

In this section, we follow the setting of Figures~\ref{fig:fdp_eval} and \ref{fig:power_eval}, and compare the $\fdr_k$ and power of a step-up procedure \cite{kfdr,kfdrfollow}, PrivateBHq, and PrivateBonf. Specifically, we consider the case of $k = 3$, and use the critical values
\[
q_j = \left(\frac{q\max\{j, k\}}{m} \prod_{i=1}^{k-1} \frac{i}{m - \max\{j, k\} + i}\right)^{\frac1k}
\]
for $j = 1, \ldots, m$, where $k = 3$ (see Eqn.~(5.3) in \cite{kfdrfollow}). For simplicity, we refer to this procedure as BHq$_k$. Figures~\ref{fig:fdp_eval_k} and \ref{fig:power_eval_k} display the results.

\begin{figure}[!htp]
\centering
\begin{subfigure}[b]{0.45\textwidth}
\centering
\includegraphics[width=\textwidth]{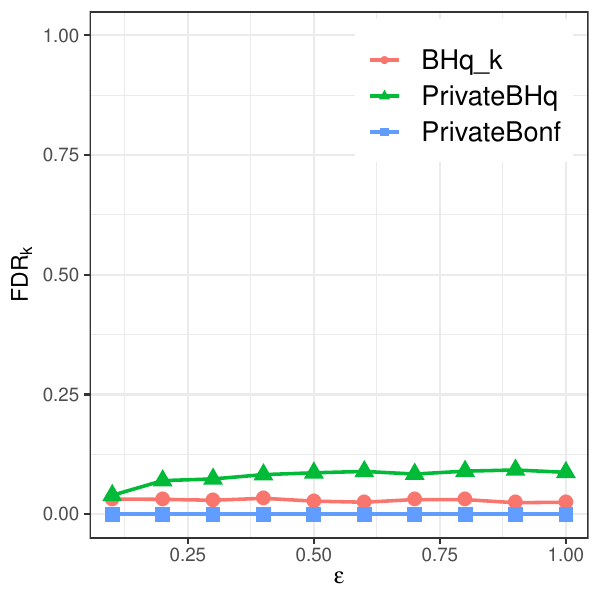}
\end{subfigure}
\begin{subfigure}[b]{0.45\textwidth}
\centering
\includegraphics[width=\textwidth]{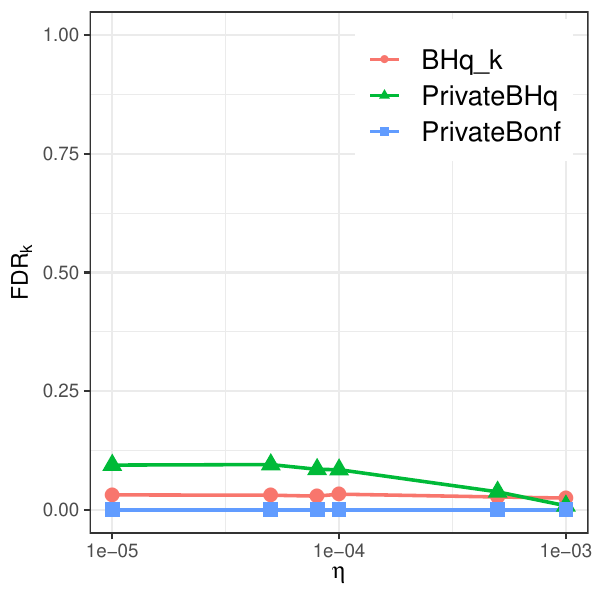}
\end{subfigure}\\
\centering
\begin{subfigure}[b]{0.45\textwidth}
\centering
\includegraphics[width=\textwidth]{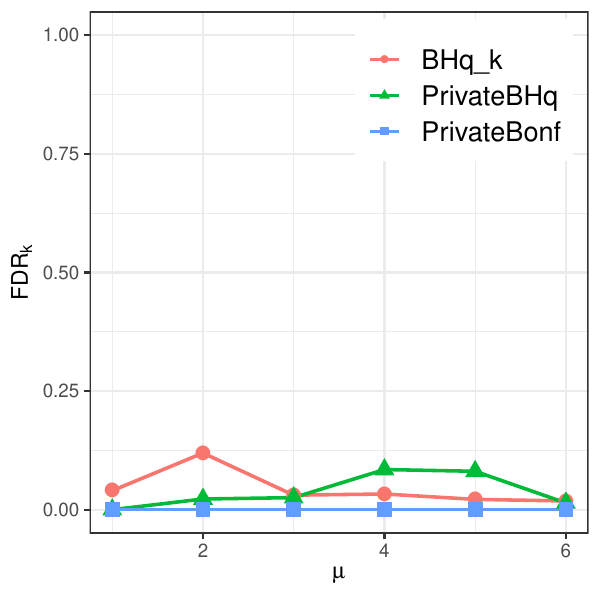}
\end{subfigure}
\begin{subfigure}[b]{0.45\textwidth}
\centering
\includegraphics[width=\textwidth]{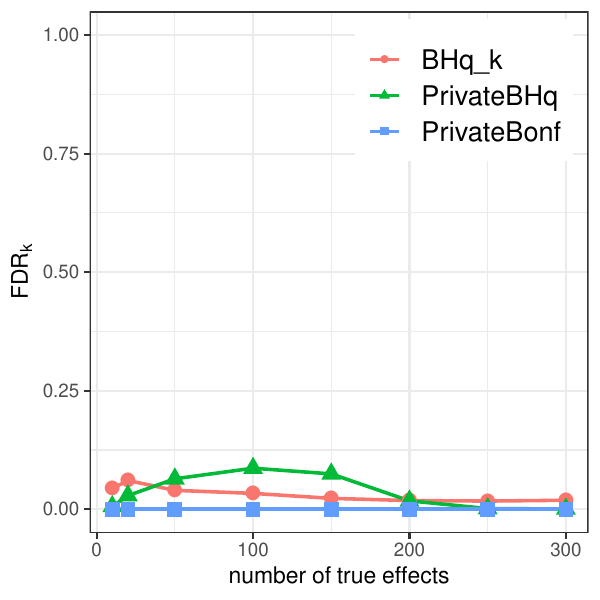}
\end{subfigure}
\hfill
\caption{The FDR of BHq$_k$, PrivateBHq, and PrivateBonf, plotted against varying $\epsilon, \eta, \mu, m_1$, respectively, and averaged over 100 independent replicates.}
\label{fig:fdp_eval_k}
\end{figure}

\begin{figure}[!htp]
\centering
\begin{subfigure}[b]{0.45\textwidth}
\centering
\includegraphics[width=\textwidth]{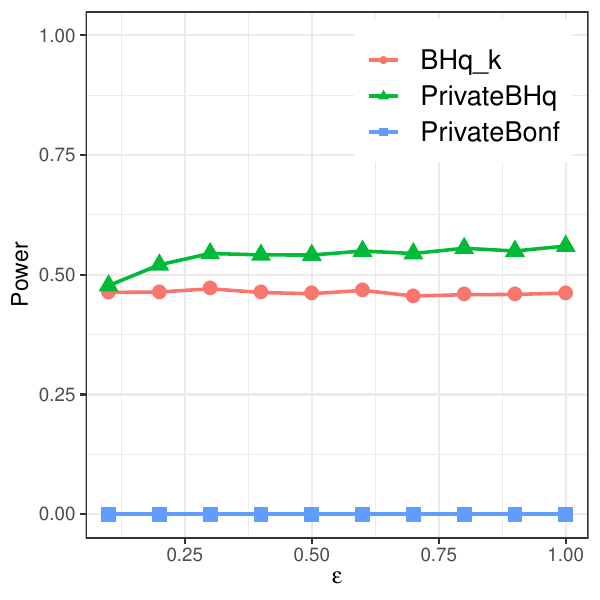}
\end{subfigure}
\begin{subfigure}[b]{0.45\textwidth}
\centering
\includegraphics[width=\textwidth]{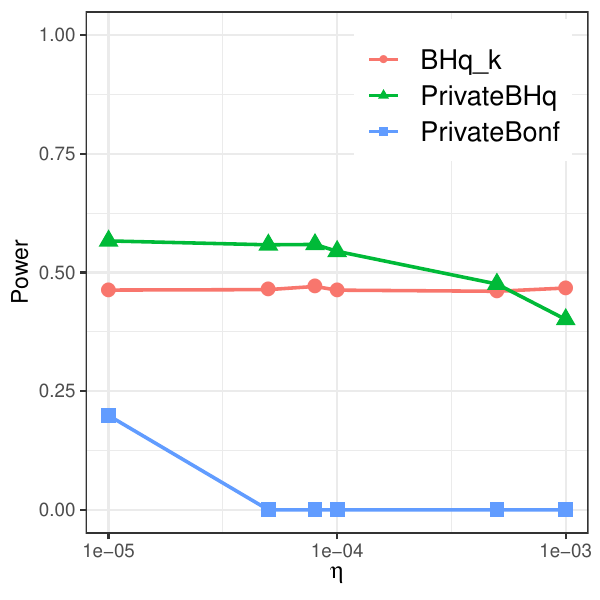}
\end{subfigure}\\
\centering
\begin{subfigure}[b]{0.45\textwidth}
\centering
\includegraphics[width=\textwidth]{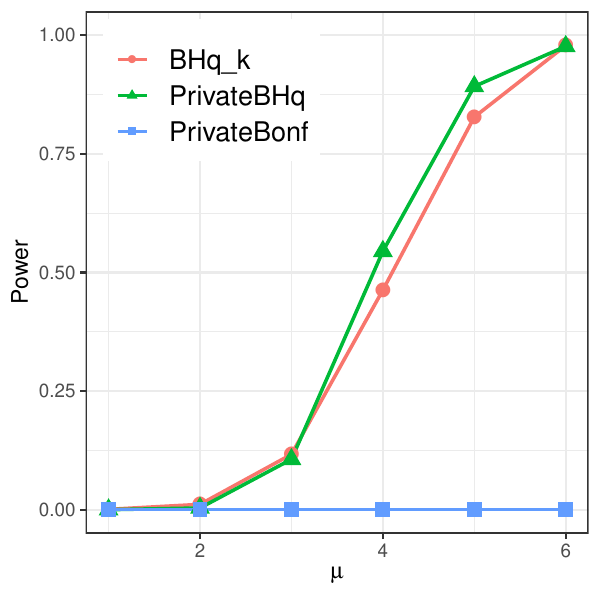}
\end{subfigure}
\begin{subfigure}[b]{0.45\textwidth}
\centering
\includegraphics[width=\textwidth]{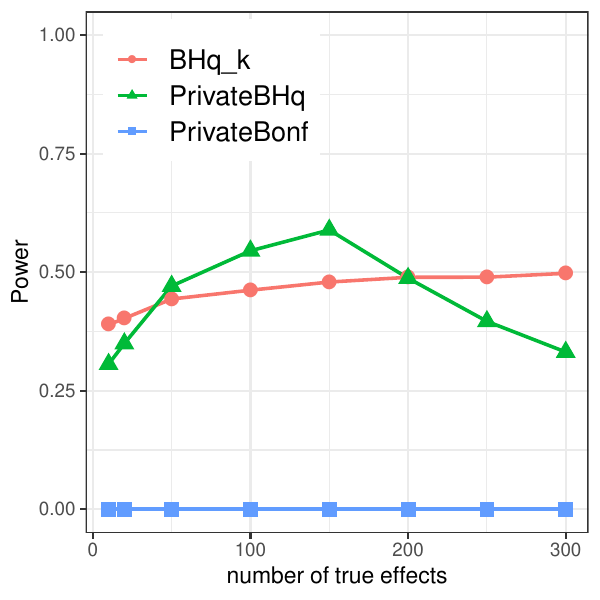}
\end{subfigure}
\hfill
\caption{The power of BHq$_k$, PrivateBHq, and PrivateBonf, plotted against varying $\epsilon, \eta, \mu, m_1$, respectively, and averaged over 100 independent replicates.}
\label{fig:power_eval_k}
\end{figure}

\subsection{BHq under Negative Dependence}
\label{sec:numerical-results}

This section features three simulated examples with certain negative
dependence between the true null and false null test statistics. The
simulation results empirically show that the BHq step-up procedure
controls $\fdr_2$ and $\fdr_5$, and this is consistent with Theorem
\ref{thm:robustfdr}. Throughout, $\mathcal{N}_0$ with cardinality $m_0$ and $\mathcal{N}_1$ with cardinality $m_1 \equiv m - m_0$ denote the set of true null hypotheses and the set of false null hypotheses, respectively.


\begin{example}[Multivariate Normal] \label{ex:normal}
Consider observing $X \sim \mathcal{N}(\mu, \Sigma)$. The covariance $\Sigma$ is constructed as follows: $\Sigma_{ii} = 1$ for all $1 \le i \le m$, $\Sigma_{ij} = 0$ if $i \ne j$ and both $i, j \in \mathcal{N}_0$ or $i, j \in \mathcal{N}_1$, and $\Sigma_{ij} = -1/\sqrt{m_0 m_1}$ if one of $i, j$ belongs to $\mathcal{N}_0$ and the other belongs to $\mathcal{N}_1$ (if this value is set to be smaller than $-1/\sqrt{m_0 m_1}$, the covariance $\Sigma$ is not positive semidefinite). The distribution of $X$ satisfies the IWN condition and, therefore, Theorem \ref{thm:robustfdr} guarantees FDR control of the BHq procedure used to test $\mu_i = 0$ against the one-sided alternative $\mu_i > 0$. In contrast, the results of \cite{prds} are not applicable because the PRDS property does not hold due to $-1/\sqrt{m_0 m_1} < 0$. Furthermore, Theorem \ref{thm:robustfdr} is still valid for testing against the two-sided alternatives $\mu_i \ne 0$. In general, the PRDS property is not satisfied for two-sided tests (see discussion in Section 3.1 of \cite{prds}).

Figure \ref{fig:normal} presents the empirical $\fdr, \fdr_2$, and $\fdr_5$ of the BHq procedure for both one-sided and two-sided alternatives in this example and, in addition, the bound $C_k \pi_0 q$ in Theorem \ref{thm:robustfdr} for $k = 2, 5$ in dashed lines. As predicted by Theorem \ref{thm:robustfdr}, the empirical $\fdr_2$ and $\fdr_5$ are indeed below their corresponding dashed lines. In fact, the empirical values are much below the bounds in Theorem \ref{thm:robustfdr}, which are derived by assuming a least favorable dependence structure between the nulls and non-nulls. This pattern is also observed in the following two plots. Moreover, these empirical error rates decrease eventually as the number of true effects $m_1$ increases, which reflects the presence of the true null proportion $\pi_0$ in the bound $C_k \pi_0 q$. Notably, this bound can be smaller than the nominal level $q$ if $m_1$ is sufficiently large.

\begin{figure}[!htp]
\centering
\begin{subfigure}[b]{0.45\textwidth}
\centering
\includegraphics[width=\textwidth]{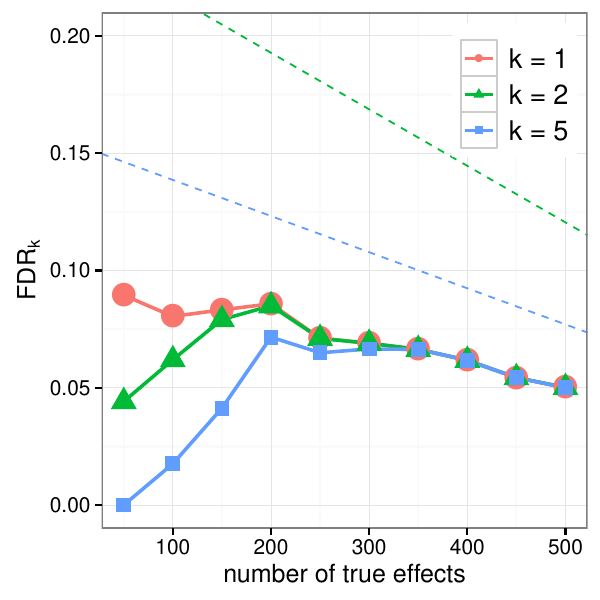}
\caption{one-sided tests}
\end{subfigure}
\begin{subfigure}[b]{0.45\textwidth}
\centering
\includegraphics[width=\textwidth]{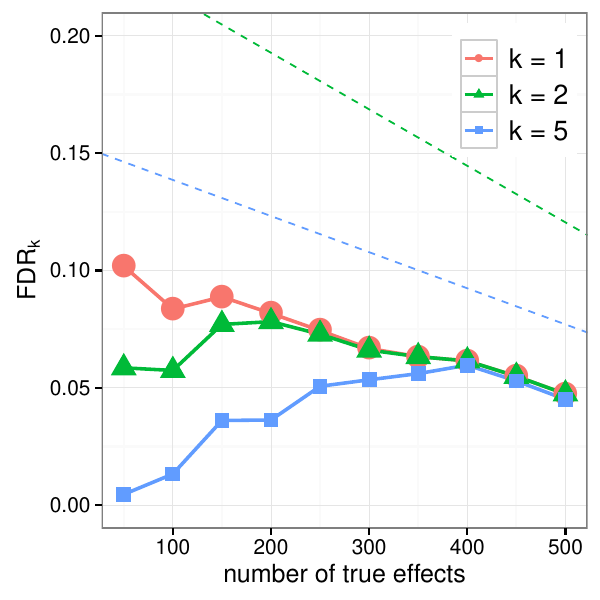}
\caption{two-sided tests}
\end{subfigure}
\hfill
\caption{$\fdr_k$ of the BHq for $k = 1, 2, 5$ in Example \ref{ex:normal}, with level $q = 0.1$. The $\fdr_1$ is just the usual FDR. We set $m$ to 1000, vary $m_1$ from 50 to 500, and let $\mu_i = 2$ for $1 \le i \le m_1$ and $\mu_i = 0$ otherwise. The covariance matrix $\Sigma$ has ones on the diagonal; $\Sigma_{ij} = \Sigma_{ji} = -1/\sqrt{m_0 m_1}$ for all $1 \le i \le m_1$ and $m_1 + 1 \le j \le m$; all the rest entries are zero. The results are averaged over 100 replicates. The upper and lower dashed lines denote the bounds $C_2 \pi_0 q$ and $C_5 \pi_0 q$, respectively.}
\label{fig:normal}
\end{figure}

\end{example}

\begin{example}[Multivariate $t$-Distribution with Different Denominators] \label{ex:student}
Consider observing \iid~vectors $X^{(1)}, \ldots, X^{(n)}$ from $\mathcal{N}(\mu, \Sigma)$, where both $\mu$ and $\Sigma$ are the same as the previous example. To test $\mu_i = 0$ against $\mu_i > 0$ or $\mu_i \ne 0$, we use the $t$-test statistics
\[
t_i = \frac{\sqrt{n}\bar X_i}{\sqrt{\frac1{n-1} \sum_{l=1}^n (X^{(l)}_i - \bar X_i)^2}},
\]
where $\bar X_i = (X^{(1)}_i + \cdots + X^{(n)}_i)/n$ for $i = 1, \ldots, m$. As earlier, Theorem \ref{thm:robustfdr} applies to this example, as opposed to the existing FDR literature, which fails to ensure FDR control of the BHq procedure in this example.

Numerical results for Example \ref{ex:student} are displayed in Figure \ref{fig:student}. The setup follows follows Example \ref{ex:normal}, with $n$ being set to 10. While the behavior of the BHq procedure in Figure \ref{fig:normal} basically remains the same in the present plot, we wish to point out that the effect of the true null proportion $\pi_0$ is more pronounced in the present simulation study and the three error rates coincide exactly once $m_1$ exceeds 100 as the BHq in this setting always rejects a substantial number of true nulls. The later shows the difference between the FDR and $\fdr_k$ is inconsequential in this example.

\begin{figure}[!htbp]
\centering
\begin{subfigure}[b]{0.45\textwidth}
\centering
\includegraphics[width=\textwidth]{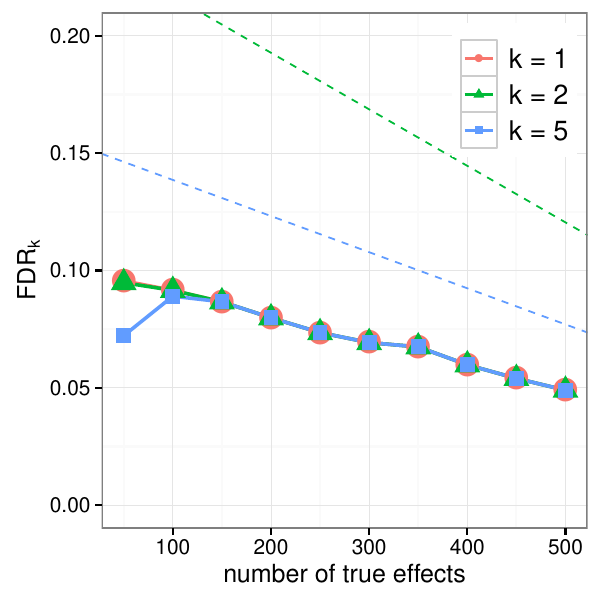}
\caption{one-sided tests}
\end{subfigure}
\begin{subfigure}[b]{0.45\textwidth}
\centering
\includegraphics[width=\textwidth]{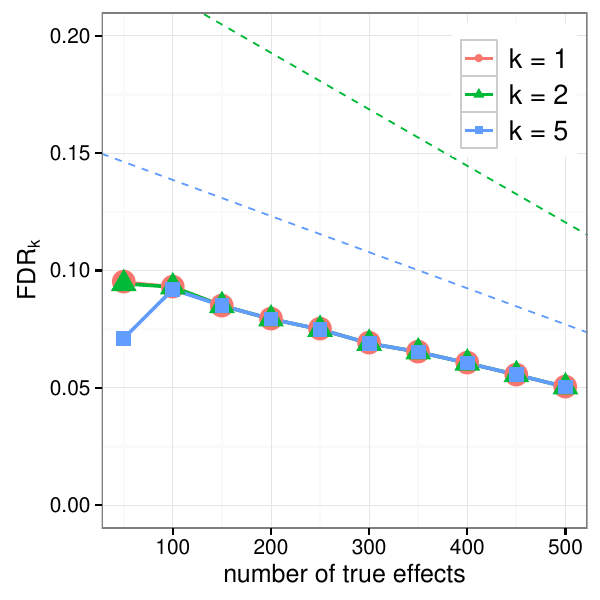}
\caption{two-sided tests}
\end{subfigure}
\hfill
\caption{$\fdr_k$ of the BHq for $k = 1, 2, 5$ in Example \ref{ex:student}, with level $q = 0.1$. The experimental setup is the same as Figure \ref{fig:normal}. The parameter $n$ is set to 10.}
\label{fig:student}
\end{figure}

\end{example}

\begin{example}[Multivariate Normal with Block-Diagonal Covariance]\label{ex:block}
Consider bivariate normal variables $X_i, \widetilde X_i$ with means $\mu_i = 0$ and $\widetilde\mu_i \ne 0$, respectively, for $i = 1, \ldots, m$. Let $\var X_i = \var \widetilde X_i = 1$ for all $i$ and the $m$ pairs $(X_i, \widetilde X_i)$ be jointly independent. Thus, the $2m$ normal variables exhibit a diagonal-block covariance matrix that is formed by $m$ $2 \times 2$ blocks on the diagonal. The correlation $\operatorname{corr}(X_i, \widetilde X_i)$ within every block varies from $-1$ to $-0.1$. Note that there are $m$ true nulls among the $2m$ hypotheses and, therefore, $\pi_0 = 0.5$. The IWN condition is satisfied because all true nulls are located in different blocks. Consequently, the BHq procedure maintains $\fdr_k$ control in this example by applying Theorem \ref{thm:robustfdr}, as opposed to existing
results in the literature, which to our knowledge are not capable of confirming the FDR control for this example. Moreover, the usual FDR control follows from Theorem \ref{thm:robustfdr} as a corollary, whose proof can be found in the appendix. As an appealing feature of this result, the dependence within each block can be arbitrary and even be different across blocks.
\begin{corollary}\label{cor:block-cov}
Fix $0 < q < 1$. Assume that $\{1 \le i \le m: \widetilde \mu_i \ge
c_1\}/ m \ge c_2$ for positive constants $c_1, c_2$ in Example
\ref{ex:block}. For both one-sided and two-sided alternatives, the BHq
procedure controls the usual FDR in an asymptotic sense. That is, as
$m \rightarrow \infty$, we get $\fdr \le (1 + o_m(1))q$.
\end{corollary}

The numerical results are summarized in Figure \ref{fig:block}. Note that the three FDR variants coincide through the range of within-block correlations. Interestingly, the bound corresponding to $k = 5$ (the lower dashed line) is below the nominal level $q = 0.1$.

\begin{figure}[!htbp]
\centering
\begin{subfigure}[b]{0.45\textwidth}
\centering
\includegraphics[width=\textwidth]{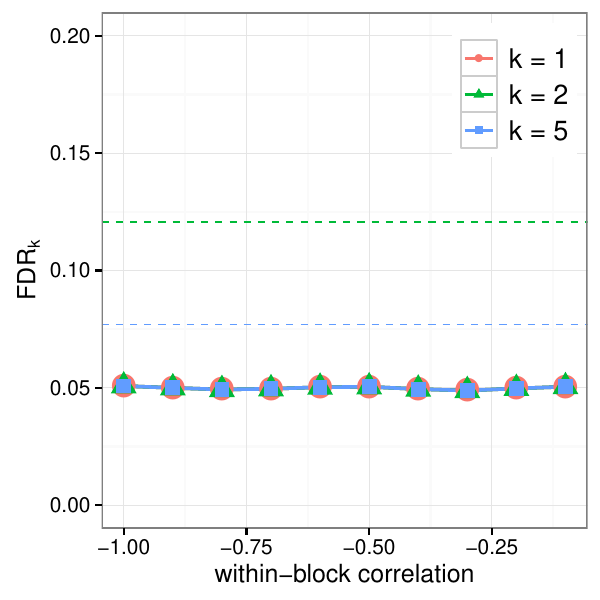}
\caption{one-sided tests}
\end{subfigure}
\begin{subfigure}[b]{0.45\textwidth}
\centering
\includegraphics[width=\textwidth]{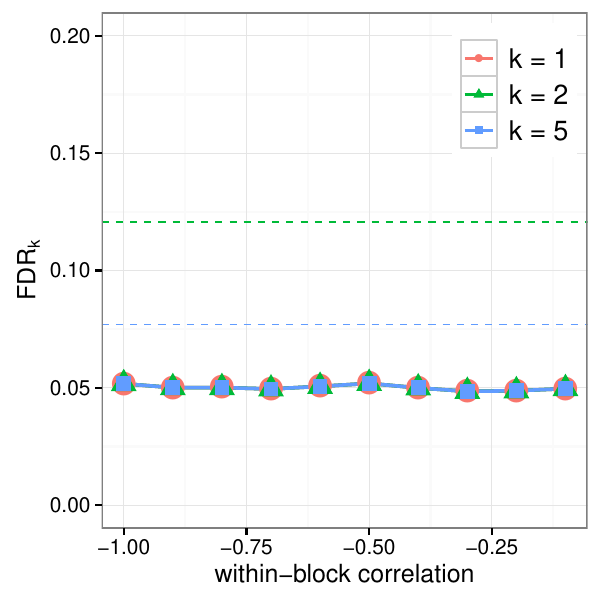}
\caption{two-sided tests}
\end{subfigure}
\hfill
\caption{$\fdr_k$ of the BHq for $k = 1, 2, 5$ in Example \ref{ex:block}, with level $q = 0.1$. Here, $m$ is set to 5000, $\widetilde \mu_i$ is set to $1.5$ for all $i$. Note that the true null proportion $\pi_0 = 0.5$. All points represent the average of 100 independent runs. The correlation between $X_i$ and $\widetilde X_i$ is set to be the same across all $i$, varying from $-1$ to $-0.1$.}
\label{fig:block}
\end{figure}

\end{example}

\begin{proof}[Proof of Corollary \ref{cor:block-cov}]
The proof idea is to apply Theorem \ref{thm:robustfdr} or Corollary \ref{cor: bhq_app} and recognize the number of rejections $R$ in Example \ref{ex:block} tends to infinity. We only consider the case of one-sided alternatives and the proof of the two-sided alternatives case is very similar.

Assume for the moment that $R \goto \infty$ with probability tending to one as $m \goto \infty$. Then, we have
\[
\begin{aligned}
|\fdp - \fdp_k| &= \frac{V \bm{1}_{V < k}}{\max\{R, 1\}} \\
& \le \frac{k - 1}{\max\{R, 1\}} \\
&= o_{\P}(1).
\end{aligned}
\]
Thus, from Theorem \ref{thm:robustfdr}, we get
\begin{equation}\label{eq:fdr_k_large}
\begin{aligned}
\fdr &= \fdr_k + o_m(1) \\
&\le C_k \pi_0 q + o_m(1) \\
&\le C_k q + o_m(1).
\end{aligned}
\end{equation}
Letting $k \goto \infty$, one gets $C_k \goto 1$. Thus, from \eqref{eq:fdr_k_large} it follows that
\[
\fdr \le q + o_m(1).
\]

Now, we aim to complete the proof by showing $R \goto \infty$ in probability. In fact, it will be shown that $\P(R \ge \lfloor \sqrt{m} \rceil) \goto 1$. By the construction of the step-up procedure, $R \ge \lfloor \sqrt{m} \rfloor$ if
\begin{equation}\nonumber
\#\left\{ i: \Phi(-X_i) \le \frac{q \lfloor \sqrt{m} \rfloor }{2m} \right\} + \#\left\{ i: \Phi(-\widetilde X_i) \le \frac{q \lfloor \sqrt{m} \rfloor }{2m} \right\} \ge \lfloor \sqrt{m} \rfloor,
\end{equation}
which is implied by
\begin{equation}\label{eq:su_suff}
\#\left\{ i: \Phi(-\widetilde X_i) \le \frac{q \lfloor \sqrt{m} \rfloor }{2m} \right\} \ge \lfloor \sqrt{m} \rfloor.
\end{equation}
Now, we aim to show \eqref{eq:su_suff} holds with probability tending to one. Denote by $A = \{i: \widetilde\mu_i \ge c_1\}$, which, by assumption, satisfies $\#A \ge c_2 m$. Consider $W_i := \widetilde X_i - \widetilde\mu_i \sim \mathcal{N}(0, 1)$ for $i \in A$ and let $W_{(1)} \ge \cdots \ge W_{(\#A)}$ be the order statistics. Note that \eqref{eq:su_suff} simply follows from
\[
\Phi \left( -c_1 - W_{(\lfloor \sqrt{m} \rfloor)} \right) \le \frac{q \lfloor \sqrt{m} \rfloor}{2m},
\]
which is equivalent to
\begin{equation}\label{eq:c_1_equiv}
-c_1 - W_{(\lfloor \sqrt{m} \rfloor)} \le \Phi^{-1} \left( \frac{q \lfloor \sqrt{m} \rfloor}{2m} \right).
\end{equation}
To prove \eqref{eq:c_1_equiv}, we make two observations:
\begin{align}
&\Phi^{-1} \left( \frac{q \lfloor \sqrt{m} \rfloor}{2m} \right) = -\sqrt{2\log\frac{2m}{q \lfloor \sqrt{m} \rfloor}} + o(1)\label{eq:normal_st_quan}\\
&W_{(\lfloor \sqrt{m} \rfloor)} = \sqrt{2\log\frac{\#A}{\lfloor \sqrt{m} \rfloor}} + o_{\P}(1)\label{eq:order_quan}
\end{align}
as both $q \lfloor \sqrt{m} \rfloor/(2m)$ tends to zero and $\#A/\lfloor \sqrt{m} \rfloor$ diverges to infinity. Above, \eqref{eq:normal_st_quan} is a standard result and a proof of \eqref{eq:order_quan} can be found in Chapter 2 of \cite{de2007extreme}. Hence, it suffices to show
\[
-\frac{c_1}{2} - \sqrt{2\log\frac{c_2 m}{\lfloor \sqrt{m} \rfloor}} \le -\sqrt{2\log\frac{2m}{q \lfloor \sqrt{m} \rfloor}}
\]
for sufficiently large $m$, which is true since
\[
\begin{aligned}
\sqrt{2\log\frac{c_2 m}{\lfloor \sqrt{m} \rfloor}} -\sqrt{2\log\frac{2m}{q \lfloor \sqrt{m} \rfloor}} &= \sqrt{\log m + 2\log c_2 + o(1)} -\sqrt{\log m + 2\log(2/q) + o(1)}\\
&= \sqrt{\log m} + o(1) -\sqrt{\log m} - o(1)\\
& = o(1).
\end{aligned}
\]

\end{proof}

\end{document}